\documentclass{amsart}
\usepackage{color}
\usepackage{amsmath}
\usepackage{amssymb}
\usepackage{amsthm}
\usepackage{amscd}
\usepackage{enumitem}
\usepackage{tikz}
\usepackage{hyperref}
\usepackage[utf8]{inputenc}
\usepackage[all]{xy}
\hypersetup{colorlinks=true}

\theoremstyle{plain}
\newtheorem{theorem}{Theorem}[section]

\newtheorem{lemma}[theorem]{Lemma}
\newtheorem{proposition}[theorem]{Proposition}
\newtheorem{corollary}[theorem]{Corollary}

\newtheorem{theoremx}{Theorem}

\theoremstyle{definition}
\newtheorem{definition}[theorem]{Definition}
\newtheorem{remark}[theorem]{Remark}

\numberwithin{equation}{section}

\newcommand\fantome[1]{}

\def\bC{\mathbb C}

\def\bK{\mathbb K}

\def\bT{\mathbb T}

\def\cO{\mathcal O}
\def\cL{\mathcal L}

\def\Fq{\mathbb F_q}
\def\lra{\longrightarrow}

\newcommand{\Ab}{\textbf{A}}

\newcommand{\Kb}{\textbf{K}}

\DeclareMathOperator{\ev}{ev}
\DeclareMathOperator{\Log}{Log}

\DeclareMathOperator{\Gal}{Gal}
\DeclareMathOperator{\End}{End}

\DeclareMathOperator{\Spec}{Spec}

\DeclareMathOperator{\sgn}{sgn}

\newcommand{\laurent}[2]{{#1 (( #2 ))}}
\newcommand{\F}{\mathbb{F}}
\newcommand{\C}{\mathbb{C}}
\newcommand{\cM}{\mathcal{M}}
\newcommand{\bA}{\textbf{A}}
\newcommand{\bu}{\mathbf{u}}
\newcommand{\bv}{\mathbf{v}}
\newcommand{\bi}{\mathbf{i}}
\newcommand{\bz}{\mathbf{z}}
\newcommand{\bb}{\mathbf{b}}
\newcommand{\bw}{\mathbf{w}}

\newcommand{\inv}{\ensuremath ^{-1}}
\newcommand{\isom}{\ensuremath \cong}
\newcommand{\tsgn}{\widetilde{\sgn}}

\newcommand{\TT}{\mathbb{T}}

\DeclareMathOperator{\Exp}{Exp}
\DeclareMathOperator{\Res}{Res}
\DeclareMathOperator{\RES}{RES}
\DeclareMathOperator{\Mat}{Mat}
\DeclareMathOperator{\divisor}{div}

\newcommand{\on}{\ensuremath ^{\otimes n}}

\newcommand{\twist}{^{(1)}}

\newcommand{\twistinv}{^{(-1)}}
\newcommand{\twisti}{^{(i)}}
\newcommand{\twistk}[1]{^{(#1)}}
\definecolor{ForestGreen}{rgb}{0.0, 0.5, 0.0}

\usepackage[small,nohug,heads=vee]{diagrams}

\author{Nathan Green}
\address{Department of Mathematics\\
University of California, San Diego (UCSD)\\
9500 Gilman Drive 0112\\
La Jolla, CA  92093-0112\\
United States of America (USA)\\}
\email{n2green@ucsd.edu}

\author{Tuan Ngo Dac}
\address{
CNRS - Université Claude Bernard Lyon 1, 
Institut Camille Jordan, UMR 5208,
43 boulevard du 11 novembre 1918,
69622 Villeurbanne Cedex, France
}
\email{ngodac@math.univ-lyon1.fr}

\title{Algebraic relations among Goss's zeta values on elliptic curves}

\begin{document}

\begin{abstract}
In 2007 Chang and Yu determined all the algebraic relations among Goss's zeta values for $A=\Fq[\theta]$ also known as the Carlitz zeta values. Goss raised the problem about algebraic relations among Goss's zeta values for a general base ring $A$ but very little is known. In this paper we develop a general method and determine all algebraic relations among Goss's zeta values for the base ring $A$ attached to an elliptic curve over $\Fq$.  To our knowledge, these are the first non-trivial solutions of Goss's problem for a base ring whose class number is strictly greater than 1.
\end{abstract}

\subjclass[2010]{Primary 11J93; Secondary 11G09, 11M38}

\keywords{Algebraic independence, Goss's zeta values, Drinfeld modules, $t$-motives, periods}

\date{\today}

\maketitle

\tableofcontents


\section*{Introduction}

\subsection{Background}
A classical topic in number theory is the study of the Riemann zeta function $\zeta(.)$ and its special values $\zeta(n)$ for $n \in \mathbb N$ and $n \geq 2$. By a well-known analogy between the arithmetic of number fields and global function fields, Carlitz suggested to transport classical results relating to the zeta function to the function field setting in positive characteristic. In \cite{Car35}, he considered the rational function field equipped with the infinity place (i.e. when $A=\Fq[\theta]$) and introduced the Carlitz zeta values $\zeta_A(n)$ which are considered as the analogues of $\zeta(n)$. Many years after Carlitz's pioneer work, Goss showed that these values could be realized as the special values of the so-called Goss-Carlitz zeta function $\zeta_A(.)$ over a suitable generalization of the complex plane. Indeed, Goss's zeta functions are a special case of the $L$-functions he introduced in \cite{Gos79} for more general base rings $A$. The special values of this type of $L$-function, called Goss's zeta values, are at the heart of function field arithmetic in the last forty years. Various works have revealed the importance of these zeta values for both their independent interest and for their applications to a wide variety of arithmetic applications, including multiple zeta values (see the excellent articles \cite{Tha16,Tha17} for an overview), Anderson's log-algebraicity identities (see \cite{And94,And96,ANDTR17a,GP18,Tha92}), Taelman's units and the class formula à la Taelman (see \cite{ANDTR20b,Deb16,Dem14,Dem15,Fan15,Mor18b,Tae12a} and \cite{ANDTR20} for an overview). 

For $A=\Fq[\theta]$, the transcendence of the Carlitz zeta values at positive integers $\zeta_A(n)$ $(n \geq 1)$ was first proved by Jing Yu \cite{Yu91}. Further, all linear and algebraic relations among these values were determined by Jing Yu \cite{Yu97} and by Chieh-Yu Chang and Jing Yu \cite{CY07}, respectively. These results are very striking when compared to the extremely limited knowledge we have about the transcendence of odd Riemann zeta values.

Goss raised the problem of extending the above work of Chang and Yu to a more general setting. For a base ring $A$ of class number one, several partial results about Goss's zeta values have been obtained by a similar method (see for example \cite{LP13}). However, to our knowledge, nothing is known when the class number of $A$ is greater than 1.

In this paper, we provide the first step towards the resolution of the above problem and develop a conceptual method to deal with the genus 1 case.  The advantage of working in the genus 1 case (elliptic curves) is that we have an explicit group law on the curve which we often exploit in our arguments.  On the other hand, where possible we strive to give general arguments in our proofs which will readily generalize to curves of arbitrary genus.  Our results  determine all algebraic relations among Goss's zeta values attached to the base ring $A$ which is the ring of regular functions of an elliptic curves over a finite field.  To do so, we reduce the study of Goss's zeta values, which are fundamentally analytic objects, to that of Anderson's zeta values, which are of arithmetic nature. Then we use a generalization of Anderson-Thakur's theorem on elliptic curves to construct zeta $t$-motives attached to Anderson's zeta values. We apply the work of Hardouin on Tannakian groups in positive characteristic and compute the Galois groups attached to zeta $t$-motives. Finally, we apply the transcendence method introduced by Papanikolas to obtain our algebraic independence result.

\bigskip

\subsection{Statement of Results}
Let us give now more precise statements of our results.

Let $X$ be a geometrically connected smooth projective curve over a finite field $\Fq$ of characteristic $p$, having $q$ elements. We denote by $K$ its  function field and fix a place $\infty$ of $K$ of degree $d_\infty = 1$. We denote by $A$ the ring of elements of $K$ which are regular outside $\infty$. The $\infty$-adic completion $K_\infty$ of $K$ is equipped with the normalized $\infty$-adic valuation $v_\infty:K_\infty \rightarrow \mathbb Z\cup\{+\infty\}$. The completion $\mathbb C_\infty$ of a fixed algebraic closure $\overline K_\infty$ of $K_\infty$ comes with a unique valuation extending $v_\infty$, it will still be denoted by $v_\infty$. 

To define Goss's zeta values (our exposition follows closely to \cite[\S 8.2-8.7]{Gos96}), we let $\pi \in K^*_\infty$ be a uniformizer so that we can identify $K_\infty$ with $\laurent{\Fq}{\pi}$. For $x\in \overline{K}_\infty^\times$, one can write $x=\pi^{v_\infty(x)} \sgn (x) \langle x \rangle$ where $\sgn(x)\in \overline{\mathbb F}_q^\times$ and $\langle x \rangle$ is a $1$-unit. If we denote by $\mathcal{I}(A)$ the group of fractional ideals of $A$, then Goss defines a group homomorphism
	\[ [\cdot]_A: \mathcal I(A) \rightarrow \overline{K}_\infty^\times \] 
such that for $x\in K^\times$,  we have $[xA]_A=x/\sgn(x)$.

Let $E/K$ be a finite extension, and let $O_E$ be the integral closure of $A$ in $E$. Then Goss defined a zeta function $\zeta_{O_E}(.)$ over a suitable generalization of the complex plane $\mathbb S_\infty$. We are interested in Goss's zeta values for $n \in \mathbb N$ given by
	\[ \zeta_{O_E}(n)=\sum_{d\geq 0} \sum_{\substack{\frak I\in \mathcal I(O_E), \frak I\subset O_E,\\ \deg(N_{E/K}(\frak I))=d}} \left[\frac{O_E}{\frak I}\right]_A^{-n} \in \overline{K}_\infty^\times \]
where $\mathcal{I}(O_E)$ denotes the group of fractional ideals of $O_E$.

\subsection{Carlitz zeta values (the genus 0 case).}

We set our curve $X$ to be the projective line $\mathbb P^1/\Fq$ equipped with the infinity point $\infty \in \mathbb P^1(\Fq)$. Then $A=\Fq[\theta]$, $K=\Fq(\theta)$ and $K_\infty=\Fq((1/\theta))$. Let $A_+$ the set of monic polynomials in $A$.

Since the class number of $A$ is $1$, by the above discussion, Goss's map is given by $[xA]_A=x/\sgn(x)$ for $x\in K^\times$. Then the Carlitz zeta values, which are special values of the Carlitz-Goss zeta function, are given by
	\[ \zeta_A(n):=\sum_{a \in A_+} \frac{1}{a^n} \in K_\infty^\times, \quad n \in \mathbb N. \]

Carlitz noticed that these values are intimately related to the so-called Carlitz module $C$ that is the first example of a Drinfeld module. Then he proved two fundamental theorems about these values. In analogy with the classical Euler formulas, Carlitz's first theorem asserts that for the so-called Carlitz period $\widetilde{\pi} \in \overline K_\infty^\times$, we have the Carliz-Euler relations:
	\[ \frac{\zeta_A(n)}{\widetilde{\pi}^n} \in K \quad \text{ for all } n\geq 1, n\equiv 0 \pmod{q-1}. \]
His second theorem states that $\zeta_A(1)$ is the logarithm of $1$ of the Carlitz module $C$, which is the first example of log-algebraicity identities.

Many years after the work of Carlitz, Anderson and Thakur \cite{AT90} developed an explicit theory of tensor powers of the Carlitz module $C^{\otimes n}$ $(n \in \mathbb N)$ and expressed $\zeta_A(n)$ as the last coordinate of the logarithm of a special algebraic point of $C^{\otimes n}$. Using this result, Yu proved that $\zeta_A(n)$ is transcendental in \cite{Yu91} and that the only $\overline K$-linear relations among the Carlitz zeta values are the above Carlitz-Euler relations in \cite{Yu97}.

For algebraic relations among the Carlitz zeta values, we obviously have the Frobenius relations which state that for $m,n \in \mathbb N$, 
	\[ \zeta_A(p^m n)=(\zeta_A(n))^{p^m}. \] 
Extending the previous works of Yu, Chang and Yu \cite{CY07} proved that the Carlitz-Euler relations and the Frobenius relations give rise to all algebraic relations among the Carlitz zeta values. To prove this result, Chang and Yu use the connection between Anderson $\Fq[\theta]$-modules and $t$-motives as well as the powerful criterion for transcendence introduced by Anderson-Brownawell-Papanikolas \cite{ABP04} and Papanikolas \cite{Pap08}. This latter criterion, which we will also use in our present paper, states roughly that the dimension of the motivic Galois group of a $t$-motive is equal to the transcendence degree of its attached period matrix. 

\subsection{Goss's zeta values on elliptic curves (the genus 1 case).} 

In a series of papers \cite{Gre17a,Gre17b,GP18}, Papanikolas and the first author  carried out an extensive study to move from the projective line $\mathbb P^1/\Fq$ (the genus 0 case) to elliptic curves over $\Fq$ (the genus 1 case). 

We work with an elliptic curve $X$ defined over $\Fq$ equipped with a rational point $\infty \in X(\Fq)$. Then $A=\Fq[\theta,\eta]$ where $\theta$ and $\eta$ satisfy a cubic Weierstrass equation for $X$. We denote by $K=\Fq(\theta,\eta)$ its fraction field and by $H \subset K_\infty$ the Hilbert class field of $A$.

The class number $\text{Cl}(A)$ of $A$ equals to the number of rational points $X(\Fq)$ on the elliptic curve $X$ and also to the degree of extension $[H:K]$, i.e.
	\[ \text{Cl}(A)=|X(\Fq)|=[H:K]. \]
For a prime ideal $\frak p$ of $A$ of degree $1$ corresponding to an $\Fq$-rational point on $X$, we consider the sum
	\[ \zeta_A(\frak p,n)=\sum_{\substack{a \in \mathfrak{p}^{-1},\\ \sgn(a)=1}}  \frac{1}{a^n}, \quad n \in \mathbb N. \]
The sums $\zeta_A(\frak p,n)$ where $\frak p$ runs through the set $\mathcal P$ of prime ideals of $A$ of degree 1 are {\it the elementary blocks} in the study of Goss's zeta values on elliptic curves. When $E=K$, $\zeta_A(n)$ can be expressed as a $\overline K$-linear combination of $\zeta_A(\frak p,n)$. When $E=H$, $\zeta_{O_H}(n)$ which is a regulator in the sense of Taelman (see \cite{ANDTR17a,Tae12a}) can be written as a product of $\overline K$-linear combinations of $\zeta_A(\frak p,n)$.

Contrary to the $\Fq[\theta]$-case, one of the main issues is that the elementary blocks $\zeta_A(\frak p,n)$ $(\frak p \in \mathcal P)$ are of analytic nature. To overcome this problem, Anderson introduced the so-called zeta values of Anderson $\zeta_\rho(b_i,n)$ (see \eqref{D:Anderson Zeta Values} for a precise definition) indexed by a $K$-basis $b_i \in O_H$ of $H$ (recall that $|\mathcal P|=[H:K]=\text{Cl}(A)$). They are also $\overline K$-linear combination of $\zeta_A(\frak p,n)$. The crucial point is that {\it Anderson's zeta values are of arithmetic nature and intimately related to the standard rank 1 sign normalized Drinfeld $A$-module $\rho$ which plays the role of the Carlitz module}.

In \cite{GP18}, Papanikolas and the first author developed an explicit theory of the above Drinfeld $A$-module $\rho$. They rediscovered the celebrated Anderson's log-algebraicity theorem on elliptic curves and proved that $\zeta_\rho(b_i,1)$ can be realized as the logarithm of $\rho$ evaluated at a prescribed algebraic point. In \cite{Gre17a,Gre17b}, the first author introduced the tensor powers $\rho\on$ for $n \in \mathbb N$ and proved basic properties of Anderson modules $\rho\on$. Then he obtained a generalization of Anderson-Thakur's theorem for small values $n <q$. By a completely different approach based on the notion of Stark units and Pellarin's $L$-series, Anglès, Tavares Ribeiro and the second author \cite{ANDTR19a} proved a generalization of Anderson-Thakur's theorem for all $n \in \mathbb N$. It states that for any $n \in \mathbb N$, Anderson's zeta values $\zeta_\rho(b_i,1)$ can be written as the last coordinate of the logarithm of $\rho\on$ evaluated at an algebraic point \footnote{In fact, this theorem holds for any general base ring $A$, see \cite{ANDTR19a}.}.

In this paper, using the aforementioned works, we generalize the work of Chang and Yu \cite{CY07} for the Carlitz zeta values and determine all algebraic relations among Anderson's zeta values on elliptic curves. 

\begin{theoremx}[Theorem \ref{theorem:main2}]
Let $m \in \mathbb N$ and $\{b_1,\ldots,b_h\}$ be a $K$-basis of $H$ with $b_i \in B$. We consider the following set
\begin{align*}
\mathcal A=\{\pi_\rho\} \cup \{\zeta_\rho(b_i,n): 1 \leq i \leq h, 1 \leq n \leq m \text{ such that } q-1 \nmid n \text{ and } p \nmid n \}
\end{align*}
where $\pi_\rho$ is the period attached to $\rho$. Then the elements of $\mathcal A$ are algebraically independent over $\overline K$.
\end{theoremx}

As an application, we also determine all algebraic relations among Goss's zeta values on elliptic curves (see also Theorem \ref{theorem: Goss zeta values}). 

\begin{theoremx}[Corollary \ref{cor: Goss zeta values}]
Let $m \in \mathbb N$ and $L$ be an extension of $K$ such that $L \subset H$. We consider the following set
\begin{align*}
\mathcal G_L=\{\pi_\rho\} \cup \{\zeta_{O_L}(n): 1 \leq n \leq m \text{ such that } q-1 \nmid n \text{ and } p \nmid n \}.
\end{align*}
Then the elements of $\mathcal G_L$ are algebraically independent over $\overline K$.
\end{theoremx}

We also prove algebraic independence of periods and logarithms of tensor powers of Drinfeld modules.

\begin{theoremx}[Theorem \ref{theorem: algebraic independence log}]
Suppose that $\mathbf{u}_1,\ldots,\mathbf{u}_m \in M_{n \times 1}(\bC_\infty)$ such that $\Exp_\rho^{\otimes n}(\mathbf{u}_i)=\mathbf{v}_i \in M_{n \times 1}(\overline K)$ and denote the $j$th entry of $\bu_i$ as $\bu_{i,j}$. If $\pi^n_\rho, \mathbf{u}_{1,n},\ldots,\mathbf{u}_{m,n}$ are linearly independent over $K$, then they are algebraically independent over $\overline K$.
\end{theoremx}

Let us sketch our proof and highlight the advances beyond \cite{CY07}. 
\begin{itemize}
\item Since we wish to apply the transcendence method of Papanikolas \cite{Pap08} (see Section \ref{sec:Papanikolas} for a summary), we will consider the $\Fq[\theta]$-modules induced by tensor powers of Drinfeld modules, still denoted by $\rho\on$ (see Section \ref{sec: background}). 

\item In Section \ref{sec: tensor powers and motives}, we construct $t$-motives attached to $\rho\on$ and compute their Galois group. Later these properties allow us to apply Hardouin's work \cite{Har11} (see Section \ref{sec:Hardouin} for a summary) which will give us the dimension of the motivic Galois group associated to the $t$-motive.

\item In Sections \ref{section: log rho} and \ref{section: log rho n}, we construct $t$-motives attached to logarithms of $\rho\on$. Our construction uses Anderson's generating functions as in \cite{CP11} instead of polygarithms used by Chang and Yu. This allows us to bypass the convergence issues of polygarithms present in \cite{CY07}. 

\item In Sections \ref{sec: period calculations} and \ref{sec: period calculations HJ}, we present two different ways to compute periods: either by direct calculations or by using a more conceptual method due to Anderson (see \cite{HJ16}, Section 5). 

\item In Section \ref{sec: Galois groups}, we compute explicitly the Galois groups of $t$-motives attached to logarithms and derive an application about algebraic independence of logarithms (see Theorem \ref{theorem: algebraic independence log}). Our calculations are completely different from all aforementioned works (e.g. \cite{CY07,CP11}) and based on a more robust method devised by Hardouin \cite{Har11}. 

\item In Section \ref{sec: Anderson zeta values}, we use a generalization of Anderson-Thakur's theorem on elliptic curves (see Theorem \ref{theorem: Anderson-Thakur}) to construct the zeta $t$-motives attached to Anderson's zeta values. Using results from Section \ref{section: log motives}, we apply the strategy of Chang-Yu to determine all algebraic relations among Anderson's zeta values (see Theorems \ref{theorem:main1} and \ref{theorem:main2}).

\item In Section \ref{sec: Goss zeta values}, we derive all algebraic relations among Goss's zeta values from those among Anderson's zeta values (see Theorem \ref{theorem: Goss zeta values} and Corollary \ref{cor: Goss zeta values}).
\end{itemize}

To summarize, we have solved completely the problem of determining all algebraic relations among Goss's zeta values on elliptic curves. Although we work on elliptic curves and make use of their group law, we have developed a general approach and expect to extend our work to a general base ring in future work. 

\bigskip

\noindent {\bf Acknowledgements.} The second author (T. ND.) was partially supported by CNRS IEA "Arithmetic and Galois extensions of function fields"; the ANR Grant COLOSS ANR-19-CE40-0015-02 and the Labex MILYON ANR-10-LABX-0070. 


\section{Background} \label{sec: background}
Traditionally, proofs in transcendental number theory tend to be quite eclectic; they pull from numerous disparate areas of mathematics.  Such is the case in this paper.  To ease the burden on the reader, we collect here a review of the various theories on which the proofs of our main theorems rely.  This review is not intended to be exhaustive and we refer the reader to various sources listed in each section.  After laying out the general notation (Section \ref{S:Notation}), we give a review of Anderson $\bA$-modules (Section \ref{S:Anderson modules}), Tensor powers of sign-normalized rank 1 Drinfeld-Hayes modules (Section \ref{S:Review of Tensor Powers}), Anderson-Thakur's theorem on zeta values and logarithms (Section \ref{S:And-Thak Theorem Gen}), linear independence of Anderson's zeta values (Section \ref{S:Linear Relations}), Papanikolas's theory on Tannakian categories and motivic Galois groups (Section \ref{sec:Papanikolas}) and Hardouin's theory on computing motivic Galois groups via the unipotent radical (Section \ref{sec:Hardouin}).
\subsection{Notation}\label{S:Notation} ${}$\par

We keep the notation of \cite{Gre17a, Gre17b, GP18} and work on elliptic curves. Throughout this paper, let $\Fq$ be a finite field of characteristic $p$, having $q$ elements. Let $X$ be an elliptic curve defined over $\Fq$ given by 
	\[ y^2+c_1ty+c_3y=t^3+c_2t^2+c_4t+c_6, \quad c_i \in \Fq. \]
It is equipped with the rational point $\infty \in X(\Fq)$ at infinity. We set $\Ab=\Fq[t,y]$ the affine coordinate ring of $X$ which is the set of functions on $X$ regular outside $\infty$ and $\Kb=\Fq(t,y)$ its fraction field. We also fix other variables $\theta, \eta$ so that $A=\Fq[\theta,\eta]$ and $K=\Fq(\theta,\eta)$ are isomorphic to $\Ab$ and $\Kb$. We denote the canonical isomorphisms $\chi: K \lra \Kb$ and $\iota: \Kb \lra K$ such that $\chi(\theta)=t$ and $\chi(\eta)=y$.

The $\infty$-adic completion $K_\infty$ of $K$ is equipped with the normalized $\infty$-adic valuation $v_\infty:K_\infty \rightarrow \mathbb Z\cup\{+\infty\}$ and has residue field $\Fq$.  We set $\deg:=-v_\infty$ so that $\deg \theta=2$ and $\deg \eta=3$. The completion $\bC_\infty$ of a fixed algebraic closure $\overline K_\infty$ of $K_\infty$ comes with a unique valuation extending $v_\infty$ which will still be denoted by $v_\infty$. We define the Frobenius $\tau: \bC_\infty \rightarrow \bC_\infty$ as the $\Fq$-algebra homomorphism which sends $x$ to $x^q$. Similarly, we can define $\Kb_\infty$ equipped with $v_\infty$ and $\deg$.

We set $\Xi=(\theta,\eta)$ which is a $K$-rational point of the elliptic curve $X$. We define a sign function $\text{\bf sgn}:\Kb_\infty^\times \rightarrow \Fq^\times$ as follows. For any $a \in \Ab$, there is a unique way to write 
	\[ a=\sum_{i \geq 0} a_i t^i + \sum_{i \geq 0} b_i t^i y, \quad a_i,b_i \in \Fq. \]
Recall that $\deg t=2$ and $\deg y=3$. The sign of $a$ is defined to the coefficient of the term of highest degree. It is easy to see that it extends to a group homomorphism 
	\[ \text{\bf sgn}:\Kb_\infty^\times \rightarrow \Fq^\times. \]
Similarly, we can define the sign function 
	\[ \sgn:K_\infty^\times \rightarrow \Fq^\times. \]

For any field extension $L/\F_q$, the coordinate ring of $E$ over $L$ is $L[t,y] = L\otimes_{\F_q} \bA$.  We extend the sign function to such rings $L[t,y] $, and using the same notion of leading term:
\[\tsgn:L(t,y)^\times \to L^\times,\]
which extends the function $\sgn$ on $\bK^\times$.
	
\subsection{Anderson $\bA$-modules on elliptic curves}\label{S:Anderson modules}  ${}$\par

We briefly review the basic theory of Anderson $\bA$-modules and dual $\bA$-motives and the relation between them.  This material follows closely to \cite[\S 3-4]{Gre17a} and the reader is directed there for proofs.

For $R$ a ring of characteristic $p$, we let $R[\tau]$ denote the (non-commutative) skew-polynomial ring with coefficients in $R$, subject to the relation for $r\in R$,
\[\tau r = r^q\tau.\]
We similarly define $R[\sigma]$, but subject to the restriction that $R$ must be an algebraically closed field and subject to the relation
\[\sigma r = r^{1/q} \sigma.\]

We define the Frobenius twist on $L[t,y]$ by setting for $g = \sum c_{j,k} t^jy^k\in  L[t,y]$, 
\begin{equation}\label{twistdef1}
g\twist = \sum c_{j,k}^q t^jy^k.
\end{equation}
The $i$th Frobenius twist is obtained by applying $i$ times the Frobenius twist. We extend twisting to matrices $\Mat_{i\times j}(L[t,y])$ by twisting coordinatewise.  We also define Frobenius twisting on points $P\in X$, also denoted by $P\twisti$, to be the the $i$th iteration of the $q$-power Frobenius isogeny.  We extend this to formal sums of points of $X$ in the natural way.  

\begin{definition}
1) An $n$-dimensional Anderson $\bA$-module is an $\F_q$-algebra homomorphism $E:\bA\to \Mat_n(\overline K_\infty)[\tau]$, such that for each $a\in \bA$, 
\[E_a = d[a] + A_1 \tau + \dots ,\quad A_i \in \Mat_n(\overline K_\infty)\]
where $d[a] = \iota(a)I + N$ for some nilpotent matrix $N\in \Mat_n(\overline K_\infty)$ (depending on $a$).  

2) A Drinfeld module is a one dimensional Andersion $\bA$-module $\rho:\bA \to \overline K_\infty[\tau]$. 
\end{definition}
We note that the map $a\mapsto d[a]$ is a ring homomorphism.

Let $E$ be an $\bA$-Anderson module of dimension $n$. We introduce the exponential and logarithm function attached to $E$, denoted $\Exp_E$ and $\Log_E$, respectively.  The exponential function is the unique function on $\C_\infty^n$ such that as an $\F_q$-linear power series we can write
\begin{equation}\label{expdef}
\Exp _E(\bz) = \sum_{i=0}^\infty Q_i \bz\twisti,\quad Q_i\in \Mat_{n}(\C_\infty), \bz\in \C_\infty^n,
\end{equation}
with $Q_0 = I_n$ and such that for all $a\in \bA$ and $\bz\in \C_\infty^n$,
\begin{equation}\label{expfunctionalequation}
\Exp_E(d[a]\bz) = E_a(\Exp_E(\bz)).
\end{equation}

The function $\Log_E$ is then defined as the formal power series inverse of $\Exp_E$.  We denote its power series as
\[
\Log _E(\bz) = \sum_{i=0}^\infty P_i \bz\twisti,\quad P_i\in \Mat_{n}(\C_\infty), \bz\in \C_\infty^n.
\]
We note that as functions on $\C_\infty^n$ the function $\Exp_E$ is everywhere convergent, whereas $\Log_E$ has some finite radius of convergence

We briefly set out some notation regarding points and divisors on the elliptic curve $X$.  We will denote addition of points on $X$ by adding the points without parenthesis, for example for $R_1,R_2 \in X$
\[R_1 + R_2 \in X,\]
and we will denote formal sums of divisors involving points on $X$ using the points inside parenthesis, for example, for $g \in K(t,y)$,
\[\divisor(g) = (R_1) - (R_2).\]
Further, multiplication on the curve $X$ will be denoted with square brackets, for example
\[[2]R_1 \in X,\]
whereas formal multiplication of points in a divisor will be denoted with simply a number where possible, or by an expression inside parenthesis, for example, for $h \in K(t,y)$,
\[\divisor(h) = 3(R_1) - (n+2)(R_2).\]

\subsection{Tensor powers of Drinfeld-Hayes modules on elliptic curves}\label{S:Review of Tensor Powers}  ${}$\par

We now construct the standard rank 1 sign-normalized Drinfeld module to which we will attach zeta values (see \cite{GP18} for a detailed account). For a general curve, we refer the interested reader to Hayes' work \cite{Hay79,Hay92}, see also \cite{And96,ANDTR17a,Thakur93} or \cite{Gos96}, Chapter 7 for more details on sign-normalized rank one Drinfeld modules.

By definition,  a rank 1 sign-normalized Drinfeld module is a Drinfeld module $\rho:\bA \to \overline K_\infty[\tau]$ such that for $a\in \bA$, we have
\[\rho_a = \iota(a) + a_1\tau + \dots + \text{\bf sgn}(a) \tau^{\deg(a)}.\]

Let $H\subset K_\infty$ be the Hilbert class field of $A$, i.e. the maximal abelian everywhere unramified extension of $K$ in which $\infty$ splits completely. There exists a unique point $V \in X(H)$ whose coordinates have positive degree, called the Drinfeld divisor (it is just a point in this situation), which verifies the equation on $X$
\[V - V\twist = \Xi.\]
We stress that $V$ is chosen to be in the formal group of $X$ at $\infty$, and in this way $V$ is uniquely determined. We recall the fact that a divisor on $X$ is principal if and only if the sum of the divisor is trivial on $X$ and the divisor has degree 0 (see \cite[Cor.~III.3.5]{Silverman}).  Thus we conclude that the divisor $(\Xi)+(V\twist) - (V) - (\infty)$ is principal and we denote the function with that divisor $f \in H(t,y)$, normalized so that $\tsgn(f)=1$, and call this the shtuka function associated to $X$ i.e.
\begin{equation}\label{D:fdiv}
\divisor(f) = (\Xi)+(V\twist) - (V) - (\infty).
\end{equation} 
We will denote the denominator and numerator of the shtuka function as
\begin{equation} \label{D:fdef}
  f := \frac{\nu(t,y)}{\delta(t)} := \frac{y - \eta - m(t-\theta)}{t-\alpha},
\end{equation}
where $m\in H$ is the slope on $E$ (in the sense of \cite[p. 53]{Silverman}) between the collinear points $V\twist, -V$ and $\Xi$, and $\deg(m) = q$, and
\begin{gather} \label{D:deltadiv}
\divisor(\nu) = (V\twist) + (-V) + (\Xi) - 3(\infty), \quad \divisor(\delta) = (V) + (-V) - 2(\infty).
\end{gather}

\begin{definition}
1) An abelian $\bA$-motive is a $\overline K[t,y,\tau]$-module $M$ which is a finitely generated projective $\overline K[t,y]$-module and free finitely generated $\overline K[\tau]$-module such that for $\ell\gg 0$ we have
	\[(t-\theta)^\ell(M/\tau M) = \{0\},\quad (y-\eta)^\ell(M/\tau M) = \{0\}.\]

2) An $\bA$-finite dual $\bA$-motive is a $\overline K[t,y,\sigma]$-module $N$ which is a finitely generated projective $\overline K[t,y]$-module and free finitely generated $\overline K[\sigma]$-module such that for $\ell\gg 0$ we have
	\[(t-\theta)^\ell(N/\sigma N) = \{0\},\quad (y-\eta)^\ell(N/\sigma N) = \{0\}.\]
Note that our definitions here are in line with \cite[\S 4.4]{BP16}, rather than the more general definition given in \cite[Def. 4.1]{HJ16}.
\end{definition}

We then let $U = \Spec \overline K[t,y]$, i.e. the affine curve $(\overline K \times_{\F_q} X) \setminus \{ \infty \}$. The (geometric) $\bA$-motive associated to $\rho$ is given by
\[
  M_1 = \Gamma(U, \cO_{X}(V))  = \bigcup_{i \geq 0} \cL((V) + i(\infty)),
\]
where $\cL((V) + i(\infty))$ is the $\overline K$-vector space of functions $g$ on $X$ with $\divisor(g) \geq -(V) - i(\infty)$.  We make $M_1$ into a left $\overline K[t,y,\tau]$-module by letting $\tau$ act by
\[
  \tau g = f g^{(1)}, \quad g \in M_1,
\]
and letting $\overline K[t,y]$ act by left multiplication.

The (geometric) dual $\bA$-motive associated to $\rho$ is given by 
\begin{equation} \label{dualt}
  N_1 = \Gamma\bigl(U, \cO_X(-(V^{(1)}))\bigr) = \bigcup_{i \geq 1} \cL(-(V\twist) + i(\infty)) \subseteq \overline K[t,y],
\end{equation}
where $\cL(-(V\twist) + i(\infty))$ is the $\overline K$-vector space of functions $g$ on $X$ with $\divisor(g) \geq (V\twist) - i(\infty)$.  We make $N_1$ into a left $\overline K[t,y,\sigma]$-module by letting $\sigma$ act by
\[
  \sigma g = f g^{(-1)}, \quad g \in N_1,
\]
and letting $\overline K[t,y]$ act by left multiplication.

We find that $M_1$ and $N_1$ are projective $\overline K[t,y]$-module of rank~$1$, that $M_1$ is as a free $\overline K[\tau]$-module of rank~$1$ and that $N_1$ is as a free $\overline K[\sigma]$-module of rank~$1$ (see \cite[\S 3]{GP18} for proofs of these facts).  A quick check shows that $M_1$ (resp. $N_1$) is indeed an abelian $\bA$-motive  (resp. $\bA$-finite dual $\bA$-motive).

We form the $n$th tensor power of $M_1$ and of $N_1$ and denote these as
\[M_n = M_1\on = M_1\otimes_{\overline K[t,y]} \dots \otimes_{\overline K[t,y]}M_1,\]
\[N_n = N_1\on = N_1\otimes_{\overline K[t,y]} \dots \otimes_{\overline K[t,y]}N_1,\]
with $\tau$ and $\sigma$ action on $a\in M_n$ and $b\in N_n$ given respectively by
\[\tau a = f^n b\twist,\quad \sigma b = f^n b\twistinv.\]

Observe that
\[M_n = \Gamma (U,\cO_X(nV)),\quad N_n \isom \Gamma (U,\cO_X(-nV\twist)),\]
and that $M_n$ (resp. $N_n$) is also an $\bA$-motive (resp. a dual $\bA$-motive).  Again, $M_n$ and $N_n$ are projective $\overline K[t,y]$-modules of rank 1. Further, $M_n$ is a free $\overline K[\tau]$-module of rank $n$ and $N_n$ is a free $\overline K[\sigma]$-module of rank $n$.

We write down convenient bases for $M_n$ and $N_n$ as a free $\overline K[\tau]$- and $\overline K[\sigma]$-modules, respectively.  Define functions $g_i \in M_n$ for $1\leq i\leq n$ with $\tsgn(g_i)=1$ and with divisors
\begin{equation}\label{gidivisor}
\divisor(g_{j}) = -n(V) +  (n-j)(\infty) + (j-1)(\Xi) + ([j-1]V\twist + [n-(j-1)]V),
\end{equation}
and similarly define functions $h_i \in N_n$ for $1\leq i\leq n$, each with $\tsgn(h_i) = 1$ and with divisor
\begin{equation}\label{hidivisor}
\divisor(h_j) = n(V\twist) - (n+j)(\infty) + (j-1)(\Xi) + (-[n - (j - 1)]V\twist-[j-1]V).
\end{equation}
Then we have
\[M_n = \overline K_\infty[\tau]\{g_1,\dots,g_n\},\quad N_n = \overline K_\infty[\sigma]\{h_1,\dots,h_n\}.\]

For $g\in N_n$ with $\deg(g) = mn + b$ with $0\leq b\leq q-1$, we define two maps
\[\delta_0,\delta_1: N_n \to \overline K_\infty^n \]
by writing $g$ in the $\overline K[\sigma]$-basis for $N_n$ described in \eqref{hidivisor},
\begin{equation}\label{gcoefficients}
g  = \sum_{i=0}^{m} \sum_{j=1}^n b_{j,i}\twistk{-i} \sigma^i h_{n-j+1},
\end{equation}
then denoting $\bb_i = (b_{1,i},b_{2,i},\dots,b_{n,i})^\top$, and setting
\begin{equation}\label{D:deltamapsdef}
\delta_0(g) = \bb_0,\quad \delta_1(g) = \bb_0 + \bb_1 + \dots + \bb_m.
\end{equation}

We then observe that $N_n/(\sigma-1)N_n \isom \overline K^n$, that the kernel of $\delta_1$ equals $\sigma-1$, and thus can write the commutative diagram
\begin{equation}\label{Amotivediagram}
\begin{diagram}
N_n/(1-\sigma)N_n &\rTo^{\delta_1} & \overline K^n  \\
\dTo^{a} & & \dTo_{\,\,\rho\on_a} \\
N_n/(1-\sigma)N_n &\rTo^{\delta_1} & \overline K^n
\end{diagram}
\end{equation}
where the left vertical arrow is multiplication by $a$ and the right vertical arrow is the map induced by multiplication by $a$, which we denote by $\rho\on_a$.

\begin{definition}\label{D:rhodef}
By \cite[Prop. 5.6]{HJ16} we know that $\rho\on$ is an Anderson $\bA$-module, which we call the $n$th tensor power of $\rho:= \rho^{\otimes 1}$.  We comment that $\rho\on$ is uniquely determined for $n\geq 1$ by the elliptic curve $X$, and that the category of dual $\bA$-motives is in anti-equivalence with the category of Anderson $\bA$-modules (with suitable conditions on each cateogry: see \cite[Thm. 5.9] {HJ16}).  The $\bA$-module $\rho\on$ is truly the $n$th tensor power of $\rho$ under this equivalence of categories.
\end{definition}

\begin{proposition}\label{P:aifacts}
We recall the following two facts about the functions $g_i$ and $h_i$ from \cite[\S 4]{Gre17a}.
\begin{enumerate}
\item For $1\leq i\leq n$, there exist constants $a_i,b_i\in H$ such that we can write 
\begin{align*}
tg_i &= \theta g_i + a_ig_{i+1} + g_{i+2},\\
th_i &= \theta h_i + b_ih_{i+1} + h_{i+2}.
\end{align*}
\item For the constants defined in \textnormal{(1)} we have $a_j = b_{n-j}$ for $1\leq j\leq n-1$ and $a_n = b_n^q.$
\end{enumerate}
\end{proposition}

We can write down the matrices defining $\rho_t\on$ using the coefficients $a_i\in H$ from Proposition \ref{P:aifacts},
\begin{equation}\label{taction}
\rho\on_t := d[\theta] + E_\theta \tau :=
\left (\begin{matrix}
\theta & a_1 & 1  & 0 & \hdots & 0 & 0 & 0 \\
0 & \theta & a_2  & 1 & \hdots & 0 & 0 & 0 \\
0 & 0 & \theta  & a_3 & \hdots & 0 & 0 & 0 \\
\vdots & \vdots &  \vdots &\vdots & \ddots &  \vdots & \vdots & \vdots \\
0 & 0 & 0  & 0 & \hdots & \theta & a_{n-2} & 1\\
0 & 0 & 0  & 0 & \hdots & 0 & \theta & a_{n-1}\\
0 & 0 & 0  & 0 & \hdots & 0 & 0 & \theta
\end{matrix}\right ) + 
\left (\begin{matrix}
0 & 0 & 0 & \hdots & 0 \\
\vdots & \vdots &  \vdots & \ddots & \vdots \\
0 & 0 & 0 & \hdots & 0\\
1 & 0 & 0 & \hdots & 0\\
a_n & 1 & 0 & \hdots & 0
\end{matrix}\right )\tau
\end{equation}

The logarithm and exponential functions associated to $\rho\on$ will be denoted $\Log_\rho\on$ and $\Exp_\rho\on$ respectively, and the period lattice of $\Exp_\rho\on$ will be denoted by $\Lambda_\rho\on$

We may consider $\rho\on$ to be a $t$-module by forgetting the action of $y$.  When it is necessary to distinguish $\rho\on$ as a $t$-module from $\rho\on$ as an $\bA$-module, we shall denote the $t$-module as $\hat \rho\on$.  When there is no confusion, we will drop the hat notation.

In \cite[Thm 5.4 and Cor. 5.6]{Gre17b} the following result is proved.
\begin{theorem}\label{T:LogCoef}
For $\bz\in \C_\infty^n$ inside the radius of convegence of $\Log_\rho\on$, if we write
\[
\Log_\rho\on(\bz) = \sum_{i=0}^\infty P_i \bz\twisti,
\]
for $n\geq 2$, then for $\lambda = \frac{dt}{2y + c_1t + c_3}$ the invariant differential on $X$,
\begin{equation}\label{Piresidue}
P_i = \left \langle \Res_\Xi\left (\frac{g_j h_{n-k+1}\twisti}{(ff\twist \dots f\twisti)^n}\lambda\right )\right  \rangle_{1\leq j,k\leq n}
\end{equation}
and $P_i \in \Mat_n(H)$ for $i\geq 0$.  Further, the bottom row of $P_i$ can be written as
\begin{equation}\label{bottomrow}
\left \langle \frac{h_{n-k+1}\twisti}{h_1(f\twist \dots f\twisti)^n}\bigg|_\Xi \right \rangle_{1\leq k\leq n}.
\end{equation}
\end{theorem}

We denote by $\mathbb T$ the Tate algebra in variable $t$ with coefficients in $\bC_\infty$ and by $\mathbb L$ the fraction field of the Tate algebra $\mathbb T$.  We now give a brief review of the functions $\omega_\rho$, $E_\bu\on$ and $G_\bu\on$ defined in \cite[\S 5-6]{Gre17a}.  Let
\begin{equation} \label{omegarhoprod}
  \omega_\rho = \xi^{1/(q-1)} \prod_{i=0}^\infty \frac{\xi^{q^i}}{f^{(i)}} \in \mathbb T[y]^\times,
\end{equation}
where $\xi\in H$ is an explicit constant given in \cite[Thm. 4.6]{GP18}.  Observe that $\omega_\rho$ satisfies the functional equation $\omega_\rho\twist = f \omega_\rho$.  For $\bu = (u_1,...,u_n)^\top \in \C_\infty^n$ define
\begin{equation}\label{Eudef}
E_{\bu}^{\otimes n}(t) = \sum_{i=0}^\infty \Exp_\rho\on\left (d[\theta]^{-i-1} \bu\right ) t^i,
\end{equation}
\begin{equation}\label{Gudef}
G_\bu\on (t,y) = E_{d[\eta]\bu}\on(t) + (y+c_1t + c_3) E_\bu\on(t).
\end{equation}
For $n=1$ and $\bu = u\in \C_\infty$, we will simplify notation by setting $E_u(t):=E_\bu^{\otimes 1}(t)$ and $G_u(t,y):=G_\bu^{\otimes 1}(t,y)$

Define $\cM$ to be the submodule of $\TT[y]$ consisting of all elements in $\TT[y]$ which have a meromorphic continuation to all of $U$.  Now define the map $\RES_{\Xi}:\cM^n\to \C_\infty^n$, for a vector of functions $(z_1,...,z_n)^\top\in \cM^n$ as
\begin{equation}\label{RESdef}
\RES_{\Xi}((z_1,\dots, z_n)^\top) = (\Res_{\Xi}(z_1\lambda),\dots, \Res_{\Xi}(z_n\lambda))^\top
\end{equation}
where $\lambda$ is the invariant differential on $E$.

We define a map $T:\TT[y]\to \TT[y]^n$ by
\begin{equation}\label{Tmapdef}
T( h(t,y)) =
\left (\begin{matrix}
h(t,y)\cdot g_1  \\
h(t,y)\cdot g_2  \\
\vdots  \\
h(t,y)\cdot g_n \\
\end{matrix}\right ).
\end{equation}

\begin{proposition} \label{P:EuReview}
We collect the following facts from \cite[\S 5-6]{Gre17a} about the above functions:
\begin{enumerate}
\item[(a)] The function $E_\bu\on \in \TT^n$ and we have the following identity of functions in $\TT^n$
\[E_\bu\on(t) = \sum_{j=0}^\infty Q_j\left (d[\theta]\twistk{j} -   tI \right )\inv\bu\twistk{j} \]
where $Q_i$ are the coefficients of $\Exp_\rho\on$ from \eqref{expdef}.
\item[(b)] The function $G_\bu\on \in \TT[y]$ and extends to a meromorphic function on $U = (\C_\infty \times_{\F_q} E) \setminus \{ \infty \}$ with poles in each coordinate only at the points $\Xi\twisti$ for $i\geq 0$.
\item[(c)] We have $\RES_\Xi(G_\bu\on) =  -(u_1,\dots,u_n )^\top$.
\item[(d)] If we denote $\Pi_n = -\RES_\Xi(T(\omega_\rho^n))$, then $T(\omega_\rho^n) = G_{\Pi_n}\on$ and the period lattice of $\Exp_\rho\on$ equals $\Lambda_\rho\on = \{d[a]\Pi_n\mid a\in \bA\}$.
\item[(e)] If $\pi_\rho$ is a fundamental period of the exponential function associated to $\rho$, and if we denote the last coordinate of $\Pi_n \in \C_\infty^n$ as $p_n$, then $p_n/\pi_\rho^n \in H$.
\end{enumerate}
\end{proposition}

\subsection{A generalization of Anderson-Thakur's theorem on elliptic curves}\label{S:And-Thak Theorem Gen} ${}$\par

Recall that $\rho: A\rightarrow \bC_\infty\{\tau\}$ is the standard sign-normalized rank one Drinfeld module on elliptic curves constructed in the previous section and that $H$ is the Hilbert class field of $A$. Let $B$ (or $\cO_H$) be the integral closure of $A$ in $H$. We denote by $G$ the Galois group ${\rm Gal}(H/K)$. 

We denote by $\mathcal{I}(A)$ the group of fractional ideals of $A$. For $I \in \mathcal I(A)$, we set
\begin{equation} \label{eq: sigma I}
\sigma_I:=(I,H/K) \in G.
\end{equation}

By \cite{Gos96}, Proposition 7.4.2 and Corollary 7.4.9, the subfield of $\mathbb C_\infty$ generated by $K$ and the coefficients of $\rho_a$ is $H$. Furthermore, by \cite{Gos96}, Lemma 7.4.5, we get
	\[ \forall a\in A, \quad \rho_a\in B\{\tau\}. \]
Let $I$ be a nonzero ideal of $A$, we define $\rho_I$ to be the unitary element in  $H\{\tau \}$ such that
	 \[ H\{ \tau \} \rho_I=\sum_{a\in I} H\{\tau\} \rho_a. \]
We have
	 \[ {\ker} \, \rho_I=\bigcap _{a\in I} {\ker}\, \rho_a, \]
	 \[ \rho_I\in B\{\tau\}, \]
	 \[ \deg_\tau \rho_I= \deg I. \]
We write $\rho_I=\rho_{I,0}+\cdots+\rho_{I,\deg I} \tau^{\deg I}$ with $\rho_{I,\deg I}=1$ and denote by $\psi(I) \in B\setminus\{0\}$ the constant coefficient $\rho_{I,0}$ of $\rho_I$. Thus the map $\psi$ extends uniquely into a map $\psi: \mathcal I(A)\rightarrow H^\times$ with the following properties:

1) for all $I, J\in \mathcal I(A), \psi(IJ)= \sigma_J(\psi(I))\, \psi(J)$,

2) for all $I\in \mathcal I(A), IB=\psi(I)B$,

3) for all $x\in K^\times, \psi (xA)= \frac{x}{\sgn (x)}$.

Finally, for $n \in \mathbb N$, we define Anderson's zeta value (also called the equivariant zeta value) at $n$ attached to $\rho$ as follows:
\begin{equation}\label{D:Anderson Zeta Values}
\zeta_\rho(b,n)=\sum_{I \subseteq A} \frac{\sigma_I(b)}{\psi(I)^n} \in K_\infty.
\end{equation}
By the works of Anderson (see \cite{And94}, \cite{And96}), for any $b \in B$, we have 
	\[ \exp_\rho(\zeta_\rho(b,1)) \in B. \]
\begin{remark}
This is an example of log-algebraicity identities for Drinfeld modules. The theory  began with the work of Carlitz \cite{Car35} in which he proved the log-algebraicity identities for the Carlitz module defined over $\Fq[\theta]$. Further examples for Drinfeld modules over $A$ which are PIDs were discovered by Thakur \cite{Tha92}. Shortly after, Anderson proved this above identity for any sign-normalized rank one Drinfeld $A$-module, known as Anderson's log-algebraicity theorem. For alternative proofs of this theorem, we refer the reader to \cite{Thakur93}, Chapter 8 for the $\Fq[\theta]$-case, \cite{GP18} for the case of elliptic curves and \cite{ANDTR17a} for the general case.
\end{remark}


\fantome{
We recall that $f$ is the shtuka function attached to $\rho$ and that for $1 \leq i \leq n$ the $h_i$ form a $\bC_\infty[\sigma]$-basis for the dual $A$-motive associated to $\rho$. Then there exist elements $d_{k,j} \in \overline K$ such that
	\[ \zeta_\rho(b,n)=\sum_{i=0}^\infty \frac{\sum_{j=0}^{\min(i,q+e)} \sum_{k=1}^n d_{k,j}^{(i)} h_{n-k+1}^{(i-j)}} {C_n h_1 (f^{(1)} \cdots f^{(i-j)})^n} \bigg|_\Xi \]
where $C_n \in H$ is an explicit constant depending on $n$. We recall the formula for the coefficients of the logarithm series for $\rho^{\otimes n}$:
	\[ \log_{\rho^{\otimes n}}=\sum_{i \geq 0} P_i \tau^i, \quad P_i \in M_{n \times n}(H). \]
It is shown that the bottom row of $P_i$ is equal to
	\[ \left( \frac{h_{n-k+1}^{(i)}}{h_1 (f^{(1)} \cdots f^{(i)})^n} \bigg|_\Xi \right)_{1 \leq k \leq n}. \]

We recall the notion of Stark units introduced in \cite{ANDTR17a,ATR17} for Drinfeld modules and \cite{ANDTR18a} for Anderson modules. For an $rn$-tuple $\alpha=(a_{1,1},a_{2,1},\ldots,a_{n,1},\ldots,a_{n,r}) \in \bC_\infty^{rn}$, we define
	\[ L_{\alpha,n}=\sum_{i=0}^\infty \frac{\sum_{j=1}^r \sum_{k=1}^n a_{k,j}^{(i)} h_{n-k+1}^{(i-j)}} {(f^{(1)} \cdots f^{(i-j)})^n} \]
where each term with $i-j<0$ is 0. Further we set
	\[ \beta_{\alpha,n}=\sum_{j=1}^r (a_{1,j}^{(j-1)}h_n+\ldots+a_{n,j}^{(j-1)} h_1)^{(-1)}. \]
Then we get
	\[ L_{\alpha,n}^{(-1)} - \frac{1}{f^n} L_{\alpha,n}=\beta_{\alpha,n}. \]

}
	
The following theorem is a generalization of the celebrated Anderson-Thakur theorem for tensor powers of the Carlitz module (see \cite{AT90}, Theorem 3.8.3).

\begin{theorem}[Anglès-Ngo Dac-Tavares Ribeiro \cite{ANDTR19a} for any $n$ and Green \cite{Gre17b} for $n<q$]  \label{theorem: Anderson-Thakur}
Let $n \geq 1$ be an integer. Then there exists a constant $C_n \in H$ such that for $b \in B$, there exists a vector $Z_n(b) \in \bC_\infty^n$ verifying the following properties:

1) We have $\Exp_\rho^{\otimes n}(Z_n(b)) \in H^n$.

2) The last coordinate of $Z_n(b)$ is equal to $C_n \zeta_\rho(b,n)$.
\end{theorem}


\subsection{Linear relations among Anderson's zeta values}\label{S:Linear Relations} ${}$\par

In this section we determine completely linear relations among Anderson's zeta values on elliptic curves. In the genus 0 case, it was done by Yu (see \cite{Yu91}, Theorem 3.1 and \cite{Yu97}, Theorem 4.1). His works are built on two main ingredients. The first one is Yu's theory where he developed an analogue of W\"ustholz's analytic subgroup theorem for function fields while the second one is the Anderson-Thakur theorem mentioned in the previous section. The main result of this section is to extend Yu's work on elliptic curves.

Recall that $A$ is generated as an $\Fq$-algebra by $t$ and $y$. Following Green (see \cite{Gre17b}, Section 7), we still denote by $\rho:\Fq[t] \lra \bC_\infty \{\tau\}$ the induced Drinfeld $\Fq[t]$-module by forgetting the $y$-action. Similarly, we denote by $\rho^{\otimes n}:\Fq[t] \lra M_n(\bC_\infty) \{\tau\}$ the Anderson $\Fq[t]$-module by forgetting the $y$-action. Basic properties of this Anderson module are given below.

\begin{proposition}[Green \cite{Gre17b}, Lemmas 7.2 and 7.3] \label{prop: simple modules}${}$\par
1) The Anderson $\Fq[t]$-module $\rho^{\otimes n}:\Fq[t] \lra M_n(\bC_\infty) \{\tau\}$ is simple in the sense of Yu (see \cite{Yu91, Yu97}). 

2) The Anderson $\Fq[t]$-module $\rho^{\otimes n}:\Fq[t] \lra M_n(\bC_\infty) \{\tau\}$ has endomorphism algebra equal to $A$. 
\end{proposition}

We slightly generalize \cite{Gre17b}, Theorem 7.1 to obtain the following theorem which settles the problem of determining linear relations among Anderson's zeta values and periods attached to $\rho$, which generalizes the work of Yu.

\begin{theorem}\label{T:LinearIndependence}
Let $\{b_1,\ldots,b_h\}$ be a $K$-basis of $H$ with $b_i \in B$. We consider the following sets for $m,s\geq 1$
\begin{align*}
& \mathcal R  :=\{\pi_\rho^k, 0 \leq k \leq m\} \cup \{\zeta_\rho(b_i,n): 1 \leq i \leq h, 1  \leq n\leq s \text{ such that } q-1 \nmid n \}, \\
& \mathcal R'  :=\{\pi_\rho^k, 0 \leq k \leq m\} \cup \{\zeta_\rho(b_i,n): 1 \leq i \leq h, 1 \leq n\leq s  \}.
\end{align*}
Then 

1) The $\overline K$-vector space generated by the elements in $\mathcal{R}$ and that generated by those in $\mathcal{R}'$ are the same. 

2) The elements in $\mathcal{R}$ are linearly independent over $\overline K$.
\end{theorem}

\begin{proof}
The proof follows the same lines as that of \cite{Yu97}, Theorem 4.1, see also \cite{Gre17b}, Theorem 7.1. We provide a proof for the convenience of the reader.

Recall that for $n \in \mathbb N$, $\rho\on$ denotes the Anderson $\Fq[t]$-module induced by $\rho\on$. We consider
\[G = G_L \times \left (\prod_{k=1}^{m} \rho^{\otimes k} \right ) \times \left (\prod_{\substack{n=1\\q-1 \nmid n}}^{s} \left( \rho^{\otimes n} \right)^h \right )\]
where $G_L$ is the trivial $t$-module.

For $1 \leq n \leq s$ set $Z_n(b_i) =(*,\ldots,*,C_n \zeta_\rho(b_i,n)^\top) \in \bC_\infty^n$ to be the vector from Theorem \ref{theorem: Anderson-Thakur} such that $\Exp_\rho\on(Z_n(b_i)) \in H^n$. For $1 \leq k \leq m$, let $\Pi_k \in \bC_\infty^k$ be a fundamental period of $\Exp_\rho^{\otimes k}$ such that the bottom coordinate of $\Pi_k$ is an $H$ multiple of $\pi_\rho^k$. Define the vector
\[ \bu = 1 \times \left (\prod_{k=1}^{m} \Pi_k \right ) \times \left (\prod_{\substack{n=1\\q-1 \nmid n}}^{s} \prod_{i=1}^h Z_n(b_i) \right ) \in G(\bC_\infty) \]
and note $\Exp_G(\bu) \in G(H)$, where $\Exp_G$ is the exponential function on $G$. Our assumption that there is a $\bar K$-linear relation among the $\zeta_\rho(b_i,n)$ and $\pi_\rho^k$ implies that $\bu$ is contained in a $d[\Fq[t]]$-invariant hyperplane of $G(\bC_\infty)$ defined over $\bar K$. This allow us to apply \cite[Thm. 1.3]{Yu97}, which says that $\bu$ lies in the tangent space to the origin of a proper $t$-submodule $G' \subset G$. Then Proposition \ref{prop: simple modules} together with [\cite[Thm. 1.3]{Yu97} imply that there exist $1 \leq n \leq s$ with $q-1 \nmid n$ and a linear relation of the form 
\[ \sum_{i=1}^h a_i \zeta_\rho(b_i,n) + b \pi_\rho^n=0 \]
for some $a_i,b \in A$ not all zero. Since $\zeta_\rho(b_i,n) \in K_\infty$ and since $H \subset K_\infty$, this implies that $b \pi_\rho^n \in K_\infty$. Since $q-1 \nmid n$, we know that $\pi_\rho^n \notin K_\infty$. It follows that $b=0$ and hence $\sum_{i=1}^h a_i \zeta_\rho(b_i,n)=0$. Since $a_i \in A$, we get
\[ 0=\sum_{i=1}^h a_i \zeta_\rho(b_i,n)=\zeta_\rho \left( \sum_{i=1}^h a_i b_i,n \right).\]
We deduce that $\sum_{i=1}^h a_i b_i=0$. Since $b_i$ is a $K$-basis of $H$, this forces $a_i=0$ for all $i$. This provides a contradiction, and proves the theorem. 
\end{proof}

As explained by B. Anglès \footnote{Personal communication in July 2019}, the following result is attributed to Carlitz and Goss  (see \cite{Gos96}, Lemma 8.22.4) which improves \cite{ANDTR17b}, Theorem 5.7 and \cite{Gre17b}, Corollary 7.4.
\begin{proposition}
Let $n \geq 1, n \equiv 0 \pmod{q-1}$ be an integer. Then for $b \in B$, we have $\zeta_\rho(b,n) / \pi_\rho^n \in K$.
\end{proposition}


\subsection{Papanikolas's work} \label{sec:Papanikolas} ${}$\par

We review Papanikolas' theory \cite{Pap08} (see also \cite{And86,ABP04}) and work with $t$-motives. Let $\overline K[t,\sigma]$ be the polynomial ring in variables $t$ and $\sigma$ with the rules
	\[at=ta, \ \sigma t=t \sigma, \ \sigma a=a^{1/q} \sigma, \quad a \in \overline K. \]
By definition, an Anderson dual $t$-motive is a left $\overline K[t,\sigma]$-module $N$ which is free and finitely generated both as a left $\overline K[t]$-module and as a left $\overline K[\sigma]$-module and which satisfies 
	\[ (t-\theta)^d N \subset \sigma N \]
for some integer $d$ sufficiently large. We consider $\overline K(t)[\sigma,\sigma\inv]$ the ring of Laurent polynomials in $\sigma$ with coefficients in $\overline K(t)$. A pre-$t$-motive is a left $\overline K(t)[\sigma,\sigma\inv]$-module that is finite dimensional over $\overline K(t)$. The category of pre-$t$-motives is abelian and there is a natural functor from the category of Anderson dual $t$-motives to the category of pre-$t$-motives
	\[ N \mapsto M:=\overline K(t) \otimes_{\overline K[t]} N \]
where $\sigma$ acts diagonally on $M$.

We now consider pre-$t$-motives $M$ which are rigid analytically trivial. Let ${\bf m} \in M_{r \times 1}(M)$ be a $\overline K(t)$-basis of $M$ and let $\Phi \in \text{GL}_r(\overline K[t])$ be the matrix representing the multiplication by $\sigma$ on $M$:
	\[ \sigma({\bf m})=\Phi {\bf m}. \] 
We recall that $\mathbb T$ is the Tate algebra in variable $t$ with coefficients in $\bC_\infty$ and that $\mathbb L$ is the fraction field of the Tate algebra $\mathbb T$. We say that $M$ is  rigid analytically trivial if there exists $\Psi \in \text{GL}_r(\mathbb L)$ such that
	\[ \Psi^{(-1)}=\Phi \Psi. \]

We set $M^\dagger:=\mathbb L \otimes_{\overline K(t)} M$ on which $\sigma$ acts diagonally and define $H_{\text{Betti}}(M)$ to be the $\Fq(t)$-submodule of $M^\dagger$ fixed by $\sigma$. We call $H_{\text{Betti}}(M)$ the Betti cohomology of $M$.	One can show that $M$ is rigid analytically trivial if and only if the natural map
	\[ \mathbb L \otimes_{\Fq(t)} H_{\text{Betti}}(M) \rightarrow M^\dagger \]
is an isomorphism. We then call $\Psi$ a rigid analytically trivialization for the matrix $\Phi$.

The category of pre-$t$-motives which are rigid analytically trivial is a neutral Tannakian category over $\Fq(t)$ with the fiber functor $\omega$ which maps $M \mapsto H_{\text{Betti}}(M)$ (see \cite{Pap08}, Theorem 3.3.15). The strictly full Tannakian subcategory generated by the images of rigid analytically trivial Anderson dual $t$-motives is called the category of $t$-motives denoted by $\mathcal T$ (see \cite{Pap08}, Section 3.4.10). By \cite{HJ16}, Remark 4.14, this categorgy is equivalent to the category of uniformizable dual $\Fq[t]$-motives given in \cite{HJ16}, Definition 4.13. 

By Tannakian duality, for each (rigid analytically trivial) $t$-motive $M$, the Tannakian subcategory generated by $M$ is equivalent to the category of finite dimensional representations over $\Fq(t)$ of some algebraic group $\Gamma_M$. It is called the (motivic) Galois group of the $t$-motive $M$. Further, we always have a faithful representation $\Gamma_M \hookrightarrow \text{GL}(H_{\text{Betti}}(M))$ which is called the tautological representation of $M$.

Papanikolas proved an analogue of Grothendieck's period conjecture which unveils a deep connection between Galois groups of $t$-motives and transcendence.

\begin{theorem}[Papanikolas \cite{Pap08}, Theorem 1.1.7] \label{theorem: Papanikolas}
Let $M$ be a $t$-motive and let $\Gamma_M$ be its Galois group. Suppose that $\Phi \in \text{GL}_n(\overline K(t)) \cap M_{n \times n}(\overline K[t])$ represents the multiplication by $\sigma$ on $M$ and that $\det \phi=c(t-\theta)^s, c \in \overline K$. If $\Psi \in \text{GL}_n(\bT)$ is a rigid analytic trivialization for $\Phi$, then
	\[ \text{tr.} \deg_{\overline K} \overline K(\Psi(\theta))=\dim \Gamma_M. \]
\end{theorem}

Papanikolas also show that $\Gamma_M$ equals to the Galois group $\Gamma_\Psi$ of the Frobenius difference equation corresponding to $M$ (see \cite{Pap08}, Theorem 4.5.10). It provides a method to compute explicitly the Galois groups for many cases of $t$-motives. It turns out to be a very powerful tool and has led to major transcendence results in the last twenty years. We refer the reader to \cite{ABP04,Cha12,CP11,CP12,CPTY10,CPY10,CPY11,CY07} for more details about transcendence applications.

Papanikolas proved that $\Gamma_M$ is an affine algebraic groupe scheme over $\Fq(t)$ which is absolutely irreducible and smooth over $\overline{\Fq(t)}$ (see \cite{Pap08}, Theorems 4.2.11, 4.3.1 and 4.5.10). Further, for any $\Fq(t)$-algebra $R$, the map 
	\[ \Gamma_M(R) \rightarrow \text{GL}(R \otimes_{\Fq(t)} H_{\text{Betti}}(M)) \] 
is given by
\begin{align} \label{eq: Galois Frob}
\gamma \mapsto (1 \otimes \Psi^{-1} {\bf m} \mapsto (\gamma^{-1} \otimes 1) \cdot (1 \otimes \Psi^{-1} {\bf m})).
\end{align}

\subsection{Hardouin's work} \label{sec:Hardouin} ${}$\par

In this section, we review the work of Hardouin \cite{Har11} on unipotent radicals of Tannakian groups in positive characteristic. Let $F$ be a field and $(\mathcal T,\omega)$ be a neutral Tannakian category over $F$ with fiber functor $\omega$. We denote by $\mathbb G_m$ the multiplicative group over $F$. For an object $\mathcal U \in \mathcal T$, we denote by $\Gamma_{\mathcal U}$ the Galois group of $\mathcal U$. Let $\mathbf 1$ be the unit object for the tensor product and $\mathcal Y$ be a completely reducible object that means $\mathcal Y$ is a direct sum of finitely many irreducible objects. We consider extensions $\mathcal U \in \text{Ext}^1(\mathbf 1,\mathcal Y)$ of $\mathbf 1$ by $\mathcal Y$, that means that we have a short exact sequence
	\[ 0 \rightarrow \mathcal Y \rightarrow \mathcal U \rightarrow \mathbf 1 \rightarrow 0. \]
For such an extension $\mathcal U$, the Galois group $\Gamma_{\mathcal U}$  of $U$ can be written as the semi-direct product
	\[ \Gamma_{\mathcal U}=R_u({\mathcal U}) \rtimes \Gamma_{\mathcal Y} \]
where $R_u({\mathcal U})$ stands for the unipotent part of $\Gamma_{\mathcal U}$. Therefore we reduce the computation of $\Gamma_{\mathcal U}$ to that of its unipotent part. In \cite{Har11}, Hardouin proves several fundamental results which characterize $R_u({\mathcal U})$ in terms of the extension group $\text{Ext}^1(\mathbf 1,\mathcal Y)$.

\begin{theorem}[Hardouin \cite{Har11}, Theorem 2] \label{theorem:Hardouin}
We keep the notation as above. Assume that
\begin{itemize}
\item[1.] every $\Gamma_{\mathcal Y}$-module is completely reducible,
\item[2.] the center of $\Gamma_{\mathcal Y}$ contains $\mathbb G_m$,
\item[3.] the action of $\mathbb G_m$ on $\omega(\mathcal Y)$ is isotypic,
\item[4.] $\Gamma_{\mathcal U}$ is reduced.
\end{itemize}
Then there exists a smallest sub-object $\mathcal V$ of $\mathcal Y$ such that $\mathcal U/\mathcal V$ is a trivial extension of $\mathbf 1$ by $\mathcal Y/\mathcal V$. Further, the unipotent part $R_u({\mathcal U})$ of the Galois group $\Gamma_{\mathcal U}$ equals to $\omega(\mathcal V)$.
\end{theorem}

As a consequence, she proves the following corollary which states that algebraic relations between the extensions are exactly given by the linear relations.

\begin{corollary}[Hardouin \cite{Har11}, Corollary 1] \label{cor:Hardouin}
We continue with the above notation. Let $\mathcal E_1,\ldots,\mathcal E_n$ be extensions of $\mathbf 1$ by $\mathcal Y$. Assume that \begin{itemize}
\item[1.] every $\Gamma_{\mathcal Y}$-module is completely reducible,
\item[2.] the center of $\Gamma_{\mathcal Y}$ contains $\mathbb G_m$,
\item[3.] the action of $\mathbb G_m$ on $\omega(\mathcal Y)$ is isotypic,
\item[4.] $\Gamma_{\mathcal E_1},\ldots,\Gamma_{\mathcal E_n}$ are reduced.
\end{itemize}
Then if $\mathcal E_1,\ldots,\mathcal E_n$ are $\End(\mathcal Y)$-linear independent in $\text{Ext}^1(\mathbf 1,\mathcal Y)$, then the unipotent radical of the Galois group $\Gamma_{\mathcal E_1 \oplus \ldots \oplus \mathcal E_n}$ of the direct sum $\mathcal E_1 \oplus \ldots \oplus \mathcal E_n$ is isomorphic to $\omega(\mathcal Y)^n$.
\end{corollary}

\section{Constructing $t$-motives connected to periods} \label{sec: tensor powers and motives}

From now on, we investigate the problem of determining algebraic relations between special zeta values and periods attached to $\rho$. For the Carlitz module, it was done by Chang and Yu \cite{CY07} using the machinery of Papanikolas \cite{Pap08} (see Section \ref{sec:Papanikolas}). For our setting, we will also need the results of Hardouin \cite{Har11} (see Section \ref{sec:Hardouin}).

\subsection{The $t$-motive associated to $\rho$} \label{sec: motive for rho} ${}$\par

We follow the construction given by Chang-Papanikolas \cite{CP11}, Sections 3.3 and 3.4. We consider $\rho:\Fq[t] \lra B\{\tau\}$ from Definition \ref{D:rhodef} as a Drinfeld $\Fq[t]$-module of rank $2$ by forgetting the $y$ action. We recall 
	\[ \rho_t=\theta+x_1 \tau + \tau^2, \quad x_1 \in B, \]
(see \cite[\S 3]{GP18} for more details on this construction).  For $u \in \bC_\infty$, we consider the associated Anderson generating function given by Equation \eqref{Eudef}:
\begin{equation*} 
E_u(t):=\sum_{n=0}^\infty \exp_\rho \left(\frac{u}{\theta^{n+1}} \right) t^n \in \mathbb{T}.
\end{equation*}
This function extends meromorphically to all of $\bC_\infty$ with simple poles at $t=\theta^{q^i}, i \geq 0$. Further, it satisfies the functional equation
	\[ \rho_t(E_u(t))=\exp_\rho(u)+t E_u(t). \]
In other words, we have
	\[ \theta E_u(t)+x_1 E_u(t)^{(1)}+E_u(t)^{(2)}=\exp_\rho(u)+t E_u(t). \]
Now we fix an $\Fq[t]$-basis $u_1 = \pi_\rho$ and $u_2 = \eta \pi_\rho$ of the period lattice $\Lambda_\rho$ of $\rho$. We set $E_i:=E_{u_i}$ for $i=1,2$. We define the following matrices:
\begin{align*}
\Phi_\rho &=
\begin{pmatrix}
0 & 1  \\
t-\theta & -x_1^{(-1)}
\end{pmatrix} \in M_{2 \times 2}(\overline K[t]), 
\quad
\Upsilon =
\begin{pmatrix}
E_1 & E_1^{(1)}  \\
E_2 & E_2^{(1)}
\end{pmatrix} \in M_{2 \times 2}(\bT),
\\
\Theta &=
\begin{pmatrix}
0 & t-\theta  \\
1 & -x_1
\end{pmatrix}  \in M_{2 \times 2}(\overline K[t]),
\quad
V =
\begin{pmatrix}
x_1 & 1  \\
1 & 0
\end{pmatrix}  \in M_{2 \times 2}(\overline K).
\end{align*}
Then we set 
	\[ \Psi_\rho:=V^{-1} (\Upsilon^{(1)})^{-1}. \]
Since $V^{(-1)} \Phi_\rho=\Theta V$ and $\Upsilon^{(1)}=\Upsilon \Theta$, we get
	\[ \Psi^{(-1)}_\rho=\Phi_\rho \Psi_\rho. \]

\subsection{The $t$-motive associated to $\rho^{\otimes n}$} \label{sec: motive for rho_n} ${}$\par

Let $n \geq 1$ be an integer. Recall the definition of $\rho^{\otimes n}$ from Definition \ref{D:rhodef}. By forgetting the $y$-action, the Anderson $\bA$-module $\rho^{\otimes n}$ can be considered as an Anderson $\Fq[t]$-module given by
	\[ \rho^{\otimes n}_t=d[\theta]+E_\theta \tau,  \]
(see section \ref{S:Review of Tensor Powers} for explicit formulas of $d[\theta]$ and $E_\theta$). For $\mathbf{u}=(u_1,\ldots,u_n)^{\top} \in M_{n \times 1}(\bC_\infty)$, we consider the associated Anderson generating function given by Equation \eqref{Eudef}:
\begin{align} \label{eq: Anderson function}
E^{\otimes n}_{\mathbf{u}}(t):=\sum_{i=0}^\infty \Exp^{\otimes n}_\rho \left(d[\theta]^{-(i+1)} \mathbf{u} \right) t^i \in \mathbb{T}^n.
\end{align}
It satisfies the functional equation
	\[ \rho_t(E^{\otimes n}_{\mathbf{u}}(t))=\Exp^{\otimes n}_\rho(\mathbf{u})+t E^{\otimes n}_{\mathbf{u}}(t). \]
As $n$ is fixed throughout this section, to simplify notation, we will suppress the dependence on $n$ and denote the coordinates of 
\[E^{\otimes n}_{\mathbf{u}}(t) := (E_{\bu,1},\dots,E_{\bu,n})^\top,\]
and similarly for other vector valued functions.  Recall the definitions of the functions $h_i$ and the coefficients $a_i$ and $b_i$ from Proposition \ref{P:aifacts} and the discussion preceding it.  For $1 \leq i \leq n$, we set 
\[
\Theta_i = \begin{pmatrix}
0 & 1\\
t-\theta &  -a_i
\end{pmatrix},\quad
\phi_i = \begin{pmatrix}
0 & 1\\
t-\theta &  -b_i
\end{pmatrix}.
\]

We have
\begin{align*}
\begin{pmatrix}
h_{i+1}  \\
h_{i+2} 
\end{pmatrix} =
\begin{pmatrix}
0 & 1  \\
t-\theta & -b_i
\end{pmatrix} 
\begin{pmatrix}
h_i  \\
h_{i+1} 
\end{pmatrix} =
\phi_i \begin{pmatrix}
h_i  \\
h_{i+1} 
\end{pmatrix}.
\end{align*}
We define
\begin{align*}
\Phi_\rho^{\otimes n}=\phi_n \ldots \phi_1 =
\begin{pmatrix}
0 & 1  \\
t-\theta & -b_n
\end{pmatrix} \cdots
\begin{pmatrix}
0 & 1  \\
t-\theta & -b_1
\end{pmatrix} \in M_{2 \times 2}(\overline K[t]).
\end{align*}
It follows that
\begin{align*}
\begin{pmatrix}
h_1  \\
h_2 
\end{pmatrix}^{(-1)} =
\begin{pmatrix}
h_{n+1}  \\
h_{n+2} 
\end{pmatrix} =
\Phi_\rho^{\otimes n}
\begin{pmatrix}
h_1  \\
h_2 
\end{pmatrix}.
\end{align*}

Now we fix an $\Fq[t]$-basis $\mathbf{u}_1 = \Pi_n$ and $\mathbf{u}_2 = d[\eta]\Pi_n$ of the period lattice $\Lambda_\rho^{\otimes n}$ of $\rho^{\otimes n}$, where $\Pi_n$ is defined in Proposition \ref{P:EuReview}(d). We denote by $E_\mathbf{i}$ the associated Anderson generating function.  In general,   Then 
	\[ \rho_t(E^{\otimes n}_{\mathbf{i}})=\Exp^{\otimes n}_\rho(\mathbf{u}_i)+t E^{\otimes n}_{\mathbf{i}}=t E^{\otimes n}_{\mathbf{i}}. \]
If we set
\begin{align*}
\Theta=\Theta_n \ldots \Theta_1=
\begin{pmatrix}
0 & 1  \\
t-\theta & -a_n
\end{pmatrix} \cdots
\begin{pmatrix}
0 & 1  \\
t-\theta & -a_1
\end{pmatrix} \in M_{2 \times 2}(\overline K[t]),
\end{align*}
and
\begin{equation}\label{D:UpsilonDef}
\Upsilon =
\begin{pmatrix}
E_{\mathbf{1},1} & E_{\mathbf{2},1}  \\
E_{\mathbf{1},2} & E_{\mathbf{2},2}
\end{pmatrix} \in M_{2 \times 2}(\mathbb{T}),
\end{equation}
then we obtain
\[
\Upsilon^{(1)}=\Theta \Upsilon.
\]
We define
\begin{align*}
V =
\begin{pmatrix}
a_n & 1  \\
1 & 0
\end{pmatrix}  \in M_{2 \times 2}(\overline K).
\end{align*}
Note that $V$ is symmetric and 
\begin{align*}
V^{(-1)} =
\begin{pmatrix}
b_n & 1  \\
1 & 0
\end{pmatrix}.
\end{align*}

We claim that
\begin{equation}\label{E:VPhiRelation}
 V^{(-1)} \Phi_\rho^{\otimes n}=(\Theta^{\top}) V.
\end{equation}
In fact, recall from Proposition \ref{P:aifacts} that $b_i=a_{n-i}$ for $1 \leq i \leq n-1$ and $a_n=b_n^q$. It is clear that for any $x \in \bC_\infty$, we have
\begin{align*}
\begin{pmatrix}
t-\theta & 0  \\
0 & 1
\end{pmatrix}
\begin{pmatrix}
0 & 1 \\
t-\theta & x
\end{pmatrix} =
\begin{pmatrix}
0 & t-\theta  \\
1 & x
\end{pmatrix}
\begin{pmatrix}
t-\theta & 0 \\
0 & 1
\end{pmatrix}.
\end{align*}
Then the claim follows immediately:
\begin{align*}
V^{(-1)} \Phi_\rho^{\otimes n} &=
\begin{pmatrix}
b_n & 1  \\
1 & 0
\end{pmatrix}
\begin{pmatrix}
0 & 1 \\
t-\theta & -b_n
\end{pmatrix} \cdots
\begin{pmatrix}
0 & 1 \\
t-\theta & -b_1
\end{pmatrix} \\
&= \begin{pmatrix}
t-\theta & 0  \\
0 & 1
\end{pmatrix}
\begin{pmatrix}
0 & 1 \\
t-\theta & -b_{n-1}
\end{pmatrix} \cdots
\begin{pmatrix}
0 & 1 \\
t-\theta & -b_1
\end{pmatrix} \\
&= \begin{pmatrix}
0 & t-\theta \\
1 & -b_{n-1}
\end{pmatrix}
\begin{pmatrix}
t-\theta & 0  \\
0 & 1
\end{pmatrix}
\begin{pmatrix}
0 & 1 \\
t-\theta & -b_{n-2}
\end{pmatrix} \cdots
\begin{pmatrix}
0 & 1 \\
t-\theta & -b_1
\end{pmatrix} \\
&= \cdots \\
&= \begin{pmatrix}
0 & t-\theta \\
1 & -b_{n-1}
\end{pmatrix} \cdots
\begin{pmatrix}
0 & t-\theta \\
1 & -b_1
\end{pmatrix}
\begin{pmatrix}
t-\theta & 0  \\
0 & 1
\end{pmatrix} \\
&= \begin{pmatrix}
0 & t-\theta \\
1 & -a_1
\end{pmatrix} \cdots
\begin{pmatrix}
0 & t-\theta \\
1 & -a_{n-1}
\end{pmatrix}
\begin{pmatrix}
t-\theta & 0  \\
0 & 1
\end{pmatrix} \\
&= \begin{pmatrix}
0 & t-\theta \\
1 & -a_1
\end{pmatrix} \cdots
\begin{pmatrix}
0 & t-\theta \\
1 & -a_{n-1}
\end{pmatrix}
\begin{pmatrix}
0 & t-\theta \\
1 & -a_n
\end{pmatrix}
\begin{pmatrix}
a_n & 1  \\
1 & 0
\end{pmatrix} \\
&= (\Theta^{\top}) V.
\end{align*}
We set
	\[ \Psi_\rho^{\otimes n}:=V^{-1} ((\Upsilon^{\top})^{(1)})^{-1}  \in M_{2 \times 2}(\mathbb{L}). \]
Thus we get
	\[ (\Psi^{\otimes n}_\rho)^{(-1)}=\Phi_\rho^{\otimes n} \Psi_\rho^{\otimes n}. \]	
	
\begin{remark} \label{remark: periods}
From the previous discussion, we have
\begin{align*}
(\Psi_\rho^{\otimes n})^{-1}=(\Upsilon^{\top})^{(1)} V=
\begin{pmatrix}
a_n E_{\mathbf{1},1}^{(1)}+E_{\mathbf{1},2}^{(1)} & E_{\mathbf{1},1}^{(1)}  \\
a_n E_{\mathbf{2},1}^{(1)}+E_{\mathbf{2},2}^{(1)} & E_{\mathbf{2},2}^{(1)}
\end{pmatrix}.
\end{align*}
By direct calculations (see Lemma \ref{L:ulemma}), we show that
	\[ [\Psi_\rho^{\otimes n}]^{-1}_{i,1}(\theta)=\mathbf{u}_{i,n} \in K \pi^n_\rho, \] 
where $\bu_{i,n}$ is the $n$th coordinate of the period $\bu_i$ for $i=1,2$.
\end{remark}

\subsection{Galois groups}	${}$\par

We denote by $\mathcal{X}^{\otimes n}_\rho$ the pre $t$-motive associated to $\rho^{\otimes n}$. The following proposition gives some basic properties of this $t$-motive (compare to \cite{CP11}, Theorem 3.5.4).

\begin{proposition} \label{proposition: Galois group}
1) The pre $t$-motive $\mathcal{X}^{\otimes n}_\rho$ is a $t$-motive.

2) The $t$-motive $\mathcal{X}^{\otimes n}_\rho$ is pure.  

3) Its Galois group $\Gamma_{\mathcal{X}_\rho^{\otimes n}}$ is $\textrm{Res}_{K/\Fq[t]} \mathbb{G}_{m,K}$. In particular, it is a torus.
\end{proposition}

\begin{proof} 
For Part 1, since $\rho^{\otimes n}$ is an $A$-motive of rank $1$, it follows from \cite{HJ16}, Proposition 3.20, Part b that $\rho^{\otimes n}$ is uniformizable as an $A$-motive. It implies that $\rho^{\otimes n}$ is uniformizable as an $\Fq[t]$-motive. Thus Part 1 follows immediately.

For Part 2, if $n=1$ then $\rho$ is a Drinfeld $\Fq[t]$-module. By \cite{HJ16}, Example 2.5, we know that $\rho$ is pure. Thus the dual $t$-motive $N$ is pure, which implies $N^{\otimes n}$ is also pure by \cite{HJ16}, Proposition 4.9, Part e. Hence we get the purity of $\mathcal{X}^{\otimes n}_\rho$.

We now prove Part 3. To calculate the Galois group $\Gamma_{\mathcal{X}_\rho^{\otimes n}}$ associated to the $t$-motive $\rho^{\otimes n}$, we will use \cite{HJ16}, Theorem 6.5. We claim that the $t$-motive $\rho^{\otimes n}$ verifies all the conditions of this theorem. In fact, it is a pure uniformizable dual $\Fq[t]$-motive thanks to Part 2. Further, it has complex multiplication since $\End_{\bC_\infty}(M)=\End_{\bC_\infty}(\rho^{\otimes n})=A$ by \cite{Gre17b}, Lemma 7.3. Thus we apply \cite{HJ16}, Theorem 6.5 to the $t$-motive $\rho^{\otimes n}$ to obtain
	\[ \Gamma_{\mathcal{X}_\rho^{\otimes n}}=\textrm{Res}_{K/\Fq[t]} \mathbb{G}_{m,K}. \]
The proof is complete.
\end{proof}
	
\begin{remark}	
1) We note that the Galois group associated to the $A$-motive $\rho^{\otimes n}$ is also equal to $\textrm{Res}_{K/\Fq[t]} \mathbb{G}_{m,K}$ (see \cite{HJ16}, Example 3.24). 

2) We should mention a similar result of Pink and his group that determines completely the Galois group of a Drinfeld $A$-module (see \cite{HJ16}, Theorem 6.3 and also \cite{CP12}, Theorem 3.5.4 for more details). It states that if $M$ is the $t$-motive associated to a Drinfeld $A$-module, then 
	\[ \Gamma_M=\text{Cent}_{\text{GL}(H_{\text{Betti}}(M))} \End_{\bC_\infty}(M). \]
\end{remark}	

\begin{proposition}
The entries of $\Psi_\rho^{\otimes n}$ are regular at $t=\theta$ and we have
	\[ \text{tr. deg}_{\overline K} \overline K(\Psi_\rho^{\otimes n}(\theta))=\dim \Gamma_{\mathcal{X}_\rho^{\otimes n}}. \]
\end{proposition}

\begin{proof}
By \cite{Pap08}, Proposition 3.3.9 (c) and Section 4.1.6, there exists a matrix $U \in \text{GL}_2(\Fq(t))$ such that $\widetilde{\Psi}=\Psi_\rho^{\otimes n} U$ is a rigid analytic trivialization of $\Phi_\rho^{\otimes n}$ and $\widetilde{\Psi} \in \text{GL}_2(\mathbb T)$. By \cite{ABP04}, Proposition 3.1.3, the entries of $\widetilde{\Psi}$ converge for all values of $\bC_\infty$. Thus the entries of $\widetilde{\Psi}$ and $U^{-1}$ are regular at $t=\theta$. It implies that the entries of $\Psi_\rho^{\otimes n}$ are regular at $t=\theta$.

Further, since $\overline K(\Psi_\rho^{\otimes n}(\theta))=\overline K(\widetilde{\Psi}(\theta))$, by Theorem \ref{theorem: Papanikolas}, we have 
	\[ \text{tr. deg}_{\overline K} \overline K(\Psi_\rho^{\otimes n}(\theta))=\text{tr. deg}_{\overline K} \overline K(\widetilde{\Psi}(\theta))=\dim \Gamma_{\mathcal{X}_\rho^{\otimes n}}. \]
The proof is finished.
\end{proof}

\subsection{Endomorphisms of $t$-motives} \label{sec: endos} ${}$\par

We write down explicitly the endomorphism of $\mathcal{X}^{\otimes n}_\rho$ given by $y \in A$. By similar arguments to \cite{Gre17a}, Proposition 4.2, there exist $y_i, z_i \in H$ such that for $1 \leq i \leq n$, we have 
	\[ yh_i=\eta h_i + y_i h_{i+1}+z_i h_{i+2}+h_{i+3}.\]
We deduce that there exists $M_y \in M_{2 \times 2}(\overline K[t])$ such that the endomorphism $y$ in the basis $\{h_1,h_2\}$ is represented by $M_y$:
\begin{align*}
y \begin{pmatrix}
h_1  \\
h_2 
\end{pmatrix} =M_y 
\begin{pmatrix}
h_1  \\
h_2 
\end{pmatrix}.
\end{align*}

More generally, we consider any element $a \in A$ as an endomorphism of $\mathcal{X}^{\otimes n}_\rho$ . We denote by $M \in M_{2 \times 2}(\overline K[t])$ such that the endomorphism $a$ in the basis $\{h_1,h_2\}$ is represented by $M$:
\begin{align*}
a \begin{pmatrix}
h_1  \\
h_2 
\end{pmatrix} =M 
\begin{pmatrix}
h_1  \\
h_2 
\end{pmatrix}.
\end{align*}

\begin{lemma}[Compare to \cite{CP11}, Proposition 4.1.1] \label{lemma:endo}
With the above notation, $M_{1,1}(\theta) \in K^\times$ and $M_{2,1}(\theta)=0$
\end{lemma}

\begin{proof}
(Compare to \cite{CP11}, Proposition 4.1.1) Specializing the above equality at $\Xi$ and recalling that $h_1(\Xi) \neq 0$ and $h_2(\Xi) = 0$, we obtain $M_{1,1}(\theta)=a(\theta) \in K^\times$ and $M_{2,1}(\theta)=0$.
\end{proof}


\section{Constructing $t$-motives connected to logarithms} \label{section: log motives}

\subsection{Logarithms attached to $\rho$} ${}$\par \label{section: log rho}

We keep the notation of Section \ref{sec: motive for rho}. Following Chang-Papanikolas (see \cite{CP11}, Section 4.2), for $u \in \bC_\infty$ with $\exp_\rho(u)=v \in \overline K$, we consider $E_u$ the associated Anderson generating function to $u$ given by \eqref{Eudef}. We set
\begin{align*}
h_v =
\begin{pmatrix}
v  \\
0
\end{pmatrix}  \in M_{2 \times 1}(\overline K),
\quad
\Phi_v =
\begin{pmatrix}
\Phi_\rho & 0  \\
h_v^{\top} & 1
\end{pmatrix} \in M_{3 \times 3}(\overline K[t]).
\end{align*}
We define
\begin{align*}
g_v &=V \begin{pmatrix}
-E_u^{(1)}  \\
-E_u^{(2)}
\end{pmatrix} = 
\begin{pmatrix}
-(t-\theta)E_u-v  \\
-E_u^{(1)}
\end{pmatrix} \in M_{2 \times 1}(\mathbb{T}), \\
\Psi_v &=
\begin{pmatrix}
\Psi_\rho & 0  \\
g_v^{\top} \Psi_\rho & 1
\end{pmatrix} \in M_{3 \times 3}(\mathbb{T}).
\end{align*}
Then we get
	\[ \Phi_\rho^{\top} g_v^{(-1)}=g_v+h_v,\]
and 
	\[ \Psi_v^{(-1)}=\Phi_v \Psi_v.\]
The associated pre-motive $\mathcal{X}_v$ is in fact a $t$-motive in the sense of Papanikolas. 

For $v:=\exp_\rho(\zeta_\rho( b,1)) \in H$, we call the corresponding $t$-motive $\mathcal{X}_v$ the zeta motive associated to $\zeta_\rho( b,1)$.

\subsection{Logarithms attached to $\rho^{\otimes n}$} ${}$\par \label{section: log rho n}

We now switch to the general case and use freely the notation of Section \ref{sec: motive for rho_n}. For $\mathbf{u} \in M_{n \times 1}(\bC_\infty)$ with $\Exp_\rho^{\otimes n}(\mathbf{u})=\mathbf{v} \in M_{n \times 1}(\overline K)$, we consider $E_\mathbf{u}\on$ the Anderson generating function associated to $\mathbf{u}$ given by Equation \eqref{eq: Anderson function}. Recall that
\begin{equation}\label{E:EuFunctionalEquation}
 \rho_t(E^{\otimes n}_{\mathbf{u}}(t))=\Exp^{\otimes n}_\rho(\mathbf{u})+t E^{\otimes n}_{\mathbf{u}}(t)=\mathbf{v}+t E^{\otimes n}_{\mathbf{u}}(t).
\end{equation}

Thus we get
\begin{align*}
\begin{pmatrix}
E_{\mathbf{u},i+1}  \\
E_{\mathbf{u},i+2} 
\end{pmatrix} =
\begin{pmatrix}
0 & 1  \\
t-\theta & -a_i
\end{pmatrix} 
\begin{pmatrix}
E_{\mathbf{u},i}  \\
E_{\mathbf{u},i+1} 
\end{pmatrix}+
\begin{pmatrix}
0  \\
-v_i 
\end{pmatrix}=
\Theta_i
\begin{pmatrix}
E_{\mathbf{u},i}  \\
E_{\mathbf{u},i+1} 
\end{pmatrix}+
\begin{pmatrix}
0  \\
-v_i 
\end{pmatrix}.
\end{align*}

We define the vector $f_\bv:= (f_{\bv,1},f_{\bv,2})^\top$ given by
\begin{equation}\label{D:fvdef}
f_\bv = \Theta_n \cdots \Theta_2  \begin{pmatrix}
0 \\
v_1
\end{pmatrix} + 
\Theta_n \cdots \Theta_3  \begin{pmatrix}
0 \\
v_2
\end{pmatrix}
+
\cdots
+
\begin{pmatrix}
0 \\
v_n
\end{pmatrix}.
\end{equation}
It follows that
\begin{align*}
\begin{pmatrix}
E_{\mathbf{u},1}  \\
E_{\mathbf{u},2} 
\end{pmatrix}^{(1)} =
\Theta
\begin{pmatrix}
E_{\mathbf{u},1}  \\
E_{\mathbf{u},2} 
\end{pmatrix}-
\begin{pmatrix}
f_{\mathbf{v},1}  \\
f_{\mathbf{v},2}
\end{pmatrix}.
\end{align*}
Here we recall 
\begin{align*}
\Theta=
\begin{pmatrix}
0 & 1  \\
t-\theta & -a_n
\end{pmatrix} \cdots
\begin{pmatrix}
0 & 1  \\
t-\theta & -a_1
\end{pmatrix}
\end{align*}
and
\begin{align*}
\Upsilon =
\begin{pmatrix}
E_{\mathbf{1},1} & E_{\mathbf{2},1}  \\
E_{\mathbf{1},2} & E_{\mathbf{2},2}
\end{pmatrix} \in M_{2 \times 2}(\mathbb{T}).
\end{align*}
They verify 
	\[ \Upsilon^{(1)}=\Theta \Upsilon.\]
We set
\begin{align*}
\Theta_\mathbf{v} =
\begin{pmatrix}
\Theta & -(f_{\mathbf{v},1},f_{\mathbf{v},2})^\top  \\
0 & 1
\end{pmatrix} \in M_{3 \times 3}(\overline K[t]),
\quad
\Upsilon_\mathbf{v} =
\begin{pmatrix}
\Upsilon & (E_{\mathbf{u},1},E_{\mathbf{u},2})^\top  \\
0 & 1
\end{pmatrix} \in M_{3 \times 3}(\mathbb{T}).
\end{align*}
Then we get
	\[ \Upsilon^{(1)}_\mathbf{v}=\Theta_\mathbf{v} \Upsilon_\mathbf{v}.\]

Now we are ready to construct the associated rigid analytic trivialization. Recall that
\begin{align*}
V =
\begin{pmatrix}
a_n & 1  \\
1 & 0
\end{pmatrix}.
\end{align*}
Note that $V$ is symmetric. We set
\begin{align*}
W &= \text{diag}(V,1)=
\begin{pmatrix}
V & 0 \\
0 & 1
\end{pmatrix}  \in \text{GL}_3(\overline K), \\
\Phi_\mathbf{v} &=(W^{(-1)})^{-1} (\Theta_\mathbf{v}^{\top}) W =\begin{pmatrix}
\Phi_\rho^{\otimes n} & 0  \\
(h_{\mathbf{v},1},h_{\mathbf{v},2}) & 1
\end{pmatrix} \in M_{3 \times 3}(\overline K[t]), \\
\Psi_\mathbf{v} &=W^{-1} ((\Upsilon_\mathbf{v}^{\top})^{(1)})^{-1}=\begin{pmatrix}
\Psi_\rho^{\otimes n} & 0  \\
(g_{\mathbf{v},1},g_{\mathbf{v},2})\Psi_\rho^{\otimes n}  & 1
\end{pmatrix} \in M_{3 \times 3}(\mathbb{T}).
\end{align*}
\begin{remark} \label{remark: periods2}
Note that by direct calculations, we obtain
	\[ (g_{\mathbf{v},1},g_{\mathbf{v},2})\Psi_\rho^{\otimes n}=-(a_n E_{\mathbf{u},1}^{(1)}+E_{\mathbf{u},2}^{(1)}, E_{\mathbf{u},1}^{(1)}).\]
Further, we will show below that $g_{\mathbf{v},1}(\theta)=u_n-v_n$, where $u_n$ and $v_n$ are the last coordinate of $\bu$ and $\bv$ respectively (see Lemma \ref{L:ulemma}).	
\end{remark}
	
Thus we get
	\[ \Psi_\mathbf{v}^{(-1)}=\Phi_\mathbf{v} \Psi_\mathbf{v}.\]	
The associated pre-motive $\mathcal{X}_\mathbf{v}$ is in fact a $t$-motive because it is an extension of two $t$-motives (see for example \cite{HJ16}, Lemma 4.20). 

\subsection{Period Calculations} \label{sec: period calculations} ${}$\par
In this section we show explicitly how to obtain the periods and zeta values discussed in the previous section from evaluations of the entries of the rigid analytic trivialization $\Psi_\bv$.
\begin{lemma}\label{L:ResEu}
Let $\bu = (u_1,\dots,u_n)^\top \in \bC_\infty^n$ and let $E_u\on$ be defined as above.  Then
\[\Res_\theta(E_\bu\on dt) = 
\left (\begin{matrix}
\Res_\theta(E_{\bu,1}dt)  \\
\vdots  \\
\Res_\theta(E_{\bu,n}dt) \\
\end{matrix}\right )
=
-\bu\]
\end{lemma}
\begin{proof}
As in the proof of \cite[Prop. 6.5]{Gre17a} we have the identity
\[E_\bu\on=  \sum_{j=0}^\infty Q_j\left (d[\theta]\twistk{j} -   tI \right )\inv\bu\twistk{j},\]
and we find that only the $j=0$ term contributes to the residue.  Then, using the cofactor expansion of $(d[\theta]\twistk{j} -   tI  )\inv$ from \cite[Pg. 26]{Gre17a} we find that
\[\Res_\theta(E_\bu\on) = 
\left (\begin{matrix}
\Res_\theta\left ((\frac{u_1}{\theta-t} + r_1(t))dt\right )\\
\vdots\\
\Res_\theta\left ((\frac{u_n}{\theta-t} + r_n(t))dt\right )
\end{matrix}\right )
=
-\left (\begin{matrix}
u_1\\
\vdots\\
u_n
\end{matrix}\right ),\]
where $r_i(t)$ is some function in powers of $(\theta-t)^k$ for $k\leq -2$, and hence does not contribute to the residue.
\end{proof}

\begin{lemma}\label{L:ulemma}
Let $E_\bu\on = (E_{\bu,1},\dots,E_{\bu,n})^\top$ as above and let $a_i$ be the defining coefficients for the Drinfeld $t$-module action as in \cite[(26)]{Gre17a}.  Then if we write $\Exp_\rho\on(\bu) := \bv = (v_1,\dots,v_n)^\top$,
\[a_nE_{\bu,1}\twist(\theta) + E_{\bu,2}\twist(\theta) = -u_n + v_n .\]
\end{lemma}
\begin{proof}
As in \eqref{E:EuFunctionalEquation}, we have that
\[\rho_t\on(E_\bu\on) = (d[\theta] + E_\theta \tau)(E_\bu\on) = \Exp_\rho\on(\bu) + tE_\bu\on.\]
Rearranging terms gives
\begin{align*}
E_\theta (E_\bu\on)\twist &= (t-d[\theta])E_\bu\on + \Exp_\rho\on(\bu).\\
\end{align*}
Both sides of the above equation are regular at $t=\theta$ in each coordinate, so then evaluating at $t=\theta$, using the formula for $E_\theta$ and taking the last coordinate gives
\[a_nE_{\bu,1}\twist(\theta) + E_{\bu,2}\twist(\theta) = E_{\bu,n}(t-\theta)\big|_{t=\theta}+v_n.\]
Finally, from our analysis in Lemma \ref{L:ResEu} we see that $E_{\bu,n}$ has a simple pole at $t=\theta$ and thus the right hand side of the above equation is simply the residue at $t=\theta$.  Thus
\[a_nE_{\bu,1}\twist(\theta) + E_{\bu,2}\twist(\theta) = \Res_{\theta}(E_{\bu,n}dt)+v_n = -u_n+v_n.\]
\end{proof}

\begin{proposition}
With all notation as above, the quantities $\pi_\rho^n$ and $u_n$ are contained in $\overline K(\Psi_\bv(\theta))$.
\end{proposition}
\begin{proof}
As above we have that $\overline K(\Psi_\bv(\theta)) = \overline K(\Upsilon_\bv(\theta))$.  Then note that 
\[E\twist_{i,j}(\theta),E\twist_{\bu,i}(\theta)\in \overline K(\Upsilon_\bv(\theta))\]
for $1\leq i,j\leq 2$ by construction.  Then by Lemma \ref{L:ulemma} we see that the last coordinates of $\bu_1,\bu_2,\bu$ are contained in $\overline K(\Upsilon_\bv(\theta))$, where $\bu_1,\bu_2$ are an $\Fq[t]$-basis for the period lattice and $\bu$ is a vector such that $\Exp_\rho\on(\bu) = \bv \in \overline K$.  From \cite[Thm. 6.7]{Gre17a} we know that the last coordinate of $\Pi_n$ is an algebraic multiple of $\pi_\rho^n$ and thus it follows that $\pi_\rho^n$ and $u_n$ are contained in $\overline K(\Psi_\bv(\theta))$.
\end{proof}

We next prove a lemma about the linear relations between the entries of $\overline K(\Psi_\rho\on(\theta))$.  Let $\pi_\rho$ and $\Pi_n$ be generators of the period lattices of $\exp_\rho$ and $\Exp_\rho\on$ respectively as $\bA$-modules.  Then recall that $\{\pi_\rho,\eta\pi_\rho\}$ and $\{\Pi_n,d[y]\Pi_n\}$ are bases for $\Lambda_\rho$ and $\Lambda_\rho\on$ respectively as $\F_q[t]$-modules.  Also for $\bu \in \C_\infty^n$, denote
\[\overline G_\bu\on = -yE_\bu\on + E_{\eta\bu}\on,\]
and note that this equals $[-1]G_\bu\on$, where $[-1]$ represents negation on the elliptic curve $X$.  We will denote the coordinates of $G_\bu\on:= (G_{\bu,1},\dots,G_{\bu,n})$ and similarly for $\overline G_{\bu}$.  

\begin{lemma}
For the field $\overline K(\Psi_\rho\on(\theta))$ we have
\begin{itemize}
\item for $n = 1$ we have $\overline K(\Psi_\rho(\theta)) = \overline K(E_{\pi_\rho}\twist(\theta),E_{\pi_\rho}\twistk{2}(\theta))$.
\item for $n\geq 2$ we have $\overline K(\Psi_\rho\on(\theta)) = \overline K(E_{\Pi_\rho,1}\twist(\theta),E_{\Pi_\rho,2}\twist(\theta))$.
\end{itemize}
Consequently, for each $n\geq 1$ we have $\overline K(\Psi_\rho\on(\theta)) = \overline K(\pi_\rho^n,Z)$ for some quantity $Z\in \C_\infty$.
\end{lemma}
\begin{proof}
For ease of exposition in the proof of this lemma, we will assume that $p=\text{char}(\Fq) \geq 3$, so that we may assume the elliptic curve $X$ has simplified Weierstrass equation given by $y^2 = t^3+at+b$ with $a,b\in \F_q$ and such that negation on $X$ is given by $(t,y)\mapsto(t,-y)$.  The lemma is also true for $p=\text{char}(\Fq) =2$, and the proof is similar, but calculations are more cumbersome.  By our definitions of $G_{\pi_\rho}$ and $G_{\Pi_n}\on$ we can write
\[
\begin{pmatrix}
y & 1 & 0 & 0\\
-y & 1 & 0 & 0\\
 0 & 0 & y & 1\\
 0 & 0 & -y & 1\\
\end{pmatrix}
\begin{pmatrix}
E_{\pi_\rho}\twist\\
E_{\eta\pi_\rho}\twist\\
E_{\pi_\rho}\twistk{2}\\
E_{\eta\pi_\rho}\twistk{2}\\
\end{pmatrix}
=
\begin{pmatrix}
G_{\pi_\rho}\twist\\
\overline G_{\pi_\rho}\twist\\
G_{\pi_\rho}\twistk{2}\\
\overline G_{\pi_\rho}\twistk{2}\\
\end{pmatrix},
\]
\[
\begin{pmatrix}
y & 1 & 0 & 0\\
-y & 1 & 0 & 0\\
 0 & 0 & y & 1\\
 0 & 0 & -y & 1\\
\end{pmatrix}
\begin{pmatrix}
E_{\Pi_n,1}\twist\\
E_{d[y]\Pi_n,1}\twist\\
E_{\Pi_n,2}\twist\\
E_{d[y]\Pi_n,2}\twist\\
\end{pmatrix}
=
\begin{pmatrix}
G_{\Pi_n,1}\twist\\
\overline G_{\Pi_n,1}\twist\\
G_{\Pi_n,2}\twist\\
\overline G_{\Pi_n,2}\twist\\
\end{pmatrix}
\]
Further, from \cite[(63)]{GP18} we get the formula $f G_{\pi_\rho} = G_{\pi_\rho}\twist$ and from Proposition \ref{P:EuReview}(d) we get the formula $g_2\cdot G_{\Pi_n,1} = g_1 \cdot G_{\Pi_n,2}$.  Note that the functions $f,g_i\in \overline K(t,y)$.  Thus, if we let $\sim$ denote a linear relation over $\overline K$ we conclude that 
\[G_{\pi_\rho}\twist(\Xi) \sim G_{\pi_\rho}\twistk{2}(\Xi),\quad \overline G_{\pi_\rho}\twist(\Xi) \sim \overline G_{\pi_\rho}\twistk{2}(\Xi),\]
\[G_{\Pi_n,1}\twist(\Xi) \sim G_{\Pi_n,2}\twist(\Xi),\quad \overline G_{\Pi_n,1}\twist(\Xi) \sim \overline G_{\Pi_n,2}\twist(\Xi).\]
This allows us to write $E_{\eta\pi_\rho}\twistk{1}(\theta)$ and $E_{\eta\pi_\rho}\twistk{2}(\theta)$ as a $\overline K$-linear combinations of $E_{\pi_\rho}\twistk{1}(\theta)$ and $E_{\pi_\rho}\twistk{2}(\theta)$ and to write $E_{d[y]\Pi_n,1}\twist(\theta)$ and $E_{d[y]\Pi_n,2}\twist(\theta)$ as a $\overline K$-linear combination of $E_{\Pi_n,1}\twist(\theta)$ and $E_{\Pi_n,2}\twist(\theta)$.  Finally, as $\overline K(\Psi_\rho\on(\theta)) = \overline K(\Upsilon(\theta))$, using Lemma \ref{L:ulemma} finishes the proof.
\end{proof}

\subsection{An application of Hartl-Juschka's work to Period Calculations} \label{sec: period calculations HJ} ${}$\par

In this section, we maintain the notation of $\bu_1,\bu_2,\bu$ from the previous section.  We wish to apply Corollaries 5.20 and 5.21 from Hartl and Juschka \cite{HJ16} to give a more conceptual method for period calculations with an aim towards generalizing these arguments to curves of arbitrary genus. We restate \cite[Cors. 5.20,5.21]{HJ16} here for the convenience of the reader, but we first translate their notation into our setting.  By definition, $\sigma$ acts by the matrix $\Phi_\rho\on$ on a $\C_\infty[t]$-basis $\{h_1,h_2\}$ of $N_n$.  Then for $z \in N_n$, we express $z = ah_1 + bh_2$ with $a,b\in \bC_\infty[t]$ and we get
	\[\sigma(z) = \sigma(ah_1 + bh_2) = \sigma\left ((a,b)
\begin{pmatrix}
h_1\\
h_2
\end{pmatrix}\right )
=
(a,b)\twistinv \Phi_\rho\on \begin{pmatrix}
h_1\\
h_2
\end{pmatrix}.\]
Thus we see that if we view $N_n$ as a free $\bC_\infty[t]$-module, then $\sigma$ acts by inverse twisting and right multiplication by $\Phi_\rho\on$ --- or we may transpose to get a left multiplication:
\[\sigma\begin{pmatrix}
a\\
b
\end{pmatrix} = \left (\Phi_\rho\on\right )^\top \begin{pmatrix}
a\\
b
\end{pmatrix}\twistinv,\quad a,b\in \F_q[t].\]
We note that this $\sigma$-action extends to $\TT^2 \isom N_n\otimes _{\F_q[t]} \TT$ in the natural way.  Recall the definition of the maps $\delta_0$ and $\delta_1$ from \eqref{D:deltamapsdef}.  These maps also extend to $N_n\otimes _{\F_q[t]} \TT$ in the natural way, which we briefly describe.  We first take an element $\bw\in \TT^2$, such that $z = \bw (h_1,h_2) \in N_n\otimes _{\F_q[t]} \TT$ and then write
\[z = c_{1,1}h_1 + \dots + c_{1,n}h_n + c_{2,1}\sigma(h_1) + \dots + c_{2,n}\sigma(h_n) + \dots.\]
Then we find that
\begin{equation}\label{D:deltaextendedmaps}
\delta_0(z) = \begin{pmatrix}
c_{1,n}\\
\vdots \\
c_{1,1}
\end{pmatrix},\quad
\delta_1(z) = \begin{pmatrix}
c_{1,n}\\
\vdots \\
c_{1,1}
\end{pmatrix} + 
\begin{pmatrix}
c_{2,n}\\
\vdots \\
c_{2,1}
\end{pmatrix}\twist
+
\cdots
.
\end{equation}

Recall from \eqref{D:fdiv} and \eqref{hidivisor} that the basis elements $\sigma^j(h_i)$ all vanish at $\Xi$ except $h_1$ itself.  Then observe that if we write $\bw = (w_1(t),w_2(t))$ and suppose that $w_1,w_2$ are regular at $\theta$, then we calculate
\[\left (w_1(t)h_1 + w_2(t)h_2\right )\big|_\Xi = w_1(\theta)h_1(\Xi) = z\big|_\Xi =  c_{1,1}h_1(\Xi).\]
Thus $w_1(\theta) = c_{1,1}$ is the bottom coordinate of $\delta_0(\bw)$.

\begin{corollary}[Hartl and Juschka, Corollaries 5.20 and 5.21]\label{L:HJLemmas}
Let $N_n$ and $\Phi_\rho\on$ be as above.  Further, let $\bw\in\TT^2$ satisfy
\begin{equation}\label{E:HJFunctionalEquation}
(\Phi_\rho\on)^\top\bw\twistinv - \bw = \bz \in N_n.
\end{equation}
Then
\[\Exp_\rho\on(\delta_0(\bw+\bz)) = \delta_1(\bz) .\]
Furthermore, if $\bz=0$, then $\delta_0(\bw) \in \Lambda_\rho\on$ and the set of all such $\bw$ form a spanning set for the periods.
\end{corollary}

So, we wish to look for vectors $\bw\in\TT^2$ which satisfy \eqref{E:HJFunctionalEquation}.
\begin{lemma}
For $\Upsilon$ as in \eqref{D:UpsilonDef}, we have
\[(\Phi_\rho\on)^\top\left (V^\top \Upsilon\twist \right )\twistinv = V^\top \Upsilon\twist.\]
\end{lemma}
\begin{proof}
From \eqref{E:VPhiRelation}, we have that $(\Phi_\rho\on) = (V\twistinv)\inv \Theta^\top V$, and then subbing in gives
\begin{align*}
(\Phi_\rho\on)^\top(V^\top \Upsilon\twist)\twistinv &= ((V\twistinv)\inv \Theta^\top V)^\top(V^\top \Upsilon\twist)\twistinv\\
&=  V^\top\Theta ((V\twistinv)\inv)^\top)(V^\top \Upsilon\twist)\twistinv\\
&= V^\top (\Theta\Upsilon)\\
&= V^\top (\Upsilon\twist).
\end{align*}
\end{proof}

We comment that $V^\top \Upsilon\twist = ((\Psi_\rho\on)\inv)^\top$, but to save on notation, we shall denote $P := V^\top \Upsilon\twist$ and denote the columns of $P$ by $P_i$.  Thus we have $(\Phi_\rho\on)^\top P_i\twistinv - P_i = 0$, and thus the vectors $P_i$ satisfy the conditions of Lemma \ref{L:HJLemmas}.  So we get
\begin{equation}\label{expPi}
\Exp_\rho\on(\delta_0(P_i+0)) = \delta_1(0) = 0.
\end{equation}

\begin{lemma}
For $\mathbf{u} \in M_{n \times 1}(\bC_\infty)$ with $\Exp_\rho^{\otimes n}(\mathbf{u})=\mathbf{v} \in M_{n \times 1}(\overline K)$, if we let $E_\bu\on$ be the Anderson generating function associated with $\bu$ and let $E_{\bu,*}$ be the first two coordinates of $E_\bu\on$, and then $P_\zeta = V^\top E_{\bu,*}\twist$ satisfies the conditions of Lemma \ref{L:HJLemmas}, that is
\[(\Phi_\rho\on)^\top(P_\zeta) - P_\zeta = V^\top f_\bv,\]
where $f_\bv$ is the vector defined in Section \ref{section: log rho n} and $V^\top f_\bv \in \overline K[t]^2$.
\end{lemma}
\begin{proof}
As stated as in Section \ref{section: log rho n} we find that $\Theta E_{\bu,*} = E_{\bu,*}\twist + f_\bv$.  We then calculate that
\begin{align*}
(\Phi_\rho\on)^\top (V^\top E_{\bu,*}\twist)\twistinv &= V^\top\Theta ((V\twistinv)\inv)^\top  (V\twistinv)^\top E_{\bu,*}\\
&= V^\top \Theta E_{\bu,*} \\
&= V^\top E_{\bu,*}\twist + V^\top f_\bv.
\end{align*}
\end{proof}

Thus $P_\zeta$ satisfies the conditions for Lemma \ref{L:HJLemmas} and we can write 
\begin{equation}\label{expPzeta}
\Exp_\rho\on(\delta_0(P_\zeta+V^\top f_\bv)) = \delta_1(V^\top f_\bv).
\end{equation}

\begin{proposition}
For a fixed $n$, with all notation as above, the quantities $\pi_\rho^n$ and $u_n$ are contained in $\overline K(\Psi_\bv(\theta))$.
\end{proposition}
\begin{proof}
By definition we have that $\overline K(\Psi_\bv(\theta)) = \overline K(\Upsilon_\bv(\theta))$, and we further see that $\overline K(\Upsilon_\bv(\theta)) = \overline K(P(\theta),E_{\mathbf{u},1}\twist(\theta),E_{\mathbf{u},2}\twist(\theta))$.  Lemma \ref{L:HJLemmas} implies that $\delta_0(P_i)$ is in the period lattice, and we deduce that the vectors $\delta_0(P_i)$ must form a generating set over $\F_q(t)$.     Finally, observing that $P_{1,i} = a_nE_{\bi,1}\twist + E_{\bi,2}\twist$ is regular at $\Xi$, we conclude that $P_{1,i}(\theta)$ gives the last coordinate $\bu_i$, our chosen generator of the $\F_q[t]$-free period lattice.  Thus we conclude that some $\overline K$-linear combination of these last coordinates gives $\pi_\rho^n$ and that $\pi_\rho^n \in \overline K(\Psi_\bv(\theta))$.

We now perform similar analysis on Equation \eqref{expPzeta}.  To proceed, we need to better understand the vector $f_\bv:= (f_{\bv,1},f_{\bv,2})^\top$.  

We recall from section \ref{sec: motive for rho_n}
\[
\Theta_i = \begin{pmatrix}
0 & 1\\
t-\theta &  a_i
\end{pmatrix},\quad
\phi_i = \begin{pmatrix}
0 & 1\\
t-\theta &  b_i
\end{pmatrix},\quad
V = \begin{pmatrix}
a_n & 1\\
1 &  0
\end{pmatrix}.
\]
If we denote
\[
X = \begin{pmatrix}
t-\theta & 0\\
0 &  1
\end{pmatrix},
\]
then we have the following equalities from \eqref{E:VPhiRelation}
\[
\Theta_n^\top V=V\twistinv \phi_n=X, \quad \Theta_{n-i}^\top X = X \phi_i, \quad 1 \leq i \leq n-1.
\]
By the definition of $f_\bv$ \eqref{D:fvdef}, 
and we find that
\[V^\top f_\bv = V^\top\Theta_n \cdots \Theta_2  \begin{pmatrix}
0 \\
v_1
\end{pmatrix} + 
V^\top\Theta_n \cdots \Theta_3  \begin{pmatrix}
0 \\
v_2
\end{pmatrix}
+
\cdots
+
V^\top\begin{pmatrix}
0 \\
v_n
\end{pmatrix}.\]
Then, in anticipation of calculating $\delta_0(V^\top f_\bv)$, using the above equalities, we find that
\begin{align*}
(h_1, h_2)V^\top f_\bv &= (h_1, h_2)\left (V^\top\Theta_n \cdots \Theta_2  \begin{pmatrix}
0 \\
v_1
\end{pmatrix} + 
V^\top\Theta_n \cdots \Theta_3  \begin{pmatrix}
0 \\
v_2
\end{pmatrix}
+
\cdots
+
V^\top\begin{pmatrix}
0 \\
v_n
\end{pmatrix}\right )\\
&= \left ( (0,v_1) \Theta_2^\top \cdots \Theta_n^\top V + \cdots + (0,v_{n-1})\Theta_n^\top V + (0,v_n) V\right )\begin{pmatrix}
h_1 \\
h_2
\end{pmatrix}\\
&= \left ( (0,v_1) X \phi_{n-2}\cdots \phi_1 + \cdots + (0,v_{n-2}) X \phi_1 +(0,v_{n-1}) X + (0,v_n) V\right )\begin{pmatrix}
h_1 \\
h_2
\end{pmatrix}.
\end{align*}
Now recall that (see Section \ref{sec: motive for rho_n}) 
\[
\begin{pmatrix}
h_{i+1}  \\
h_{i+2} 
\end{pmatrix} =
\phi_i \begin{pmatrix}
h_i  \\
h_{i+1} 
\end{pmatrix}.
\]
Since $(0,v_i)X=(0,v_i)$, it follows that
\begin{align*}
(h_1, h_2)V^\top f_\bv &= \left ( (0,v_1) X \phi_{n-2}\cdots \phi_1 + \cdots + (0,v_{n-2}) X \phi_1 +(0,v_{n-1}) X + (0,v_n) V\right )\begin{pmatrix}
h_1 \\
h_2
\end{pmatrix}\\
&= (0,v_1)\begin{pmatrix}
h_{n-1} \\
h_{n}
\end{pmatrix} +
\cdots
+ (0,v_{n-2})\begin{pmatrix}
h_2 \\
h_3
\end{pmatrix}
+ (0,v_{n-1})\begin{pmatrix}
h_1 \\
h_2
\end{pmatrix} +
(v_{n},0)\begin{pmatrix}
h_1 \\
h_2
\end{pmatrix}\\
&= v_1h_n + \cdots + v_nh_1.
\end{align*}
Thus, we find that
\[\delta_0(V^\top f_\bv) = \delta_1(V^\top f_\bv)  =  \begin{pmatrix}
v_1 \\
\vdots \\
v_n
\end{pmatrix}= \bv.\]
Then, returning to \eqref{expPzeta} we find that
\[\Exp_\rho\on(\delta_0(P_\zeta) + \bv) = \bv = \Exp_\rho\on(\bu).\]
Thus the two quantities in the exponential functions in the above equality differ by a period, so there exists some $a\in A$ such that
\[\delta_0(P_\zeta) + \bv = \bu + d[a]\Pi_n,\]
where $d[a]$ denotes the action of $a$ under $\text{Lie}(\rho\on)$.  By our above analysis of $\delta_0$, we conclude that the bottom coordinate $\alpha$ of $\delta_0(P_\zeta)$ is equal to
\[\alpha = a_nE_{\bu,1}\twist(\theta) + E_{\bu,2}\twist(\theta) =  u_n-v_n + a \pi_\rho^n.\]
Since we have proved above that $\pi_\rho^n \in \overline K(\Psi_\bv(\theta))$ and since $v_n \in \overline K$, it follows that $u_n\in \overline K(\Psi_\bv(\theta))$ as well.
\end{proof}

\subsection{Independence in $\text{Ext}^1_\mathcal{T}(\mathbf 1,\mathcal{X}_\rho^{\otimes n})$} ${}$\par

For $b\in B$ and $\mathbf{v}:=\Exp_\rho^{\otimes n}(Z_n(b)) \in M_{n \times 1}(H)$, where $Z_n(b)$ is the log-algebraic vector of Theorem \ref{theorem: Anderson-Thakur}, we call the corresponding $t$-motive $\mathcal{X}_n(b):=\mathcal{X}_\mathbf{v}$ the zeta $t$-motive associated to $\zeta_\rho( b,n)$. We also denote by $\Phi_n(b)$ and $\Psi_n(b)$ the corresponding matrices. We have a short exact sequence
	\[ 0 \lra \mathcal{X}^{\otimes n}_\rho \lra \mathcal{X}_n(b) \lra \mathbf{1} \lra 0\]
where $\mathbf{1}$ is the trivial $t$-motive.

We follow closely \cite{CP12}, Section 4.2. The group $\text{Ext}^1_\mathcal{T}(\mathbf 1,\mathcal{X}_\rho^{\otimes n})$ has the structure of a $K$-vector space by pushing along $\mathcal{X}_\rho^{\otimes n}$. Let $\mathcal{X} \in \text{Ext}^1_\mathcal{T}(\mathbf 1,\mathcal{X}_\rho^{\otimes n})$ be an extension $\mathcal{X}$ represented by $\mathbf{v} \in M_{1 \times 2}(\overline K)$ and $a \in A$ whose corresponding matrix is $E \in M_{2 \times 2}(\overline K[t])$, then the extension $a_* \mathcal{X}$ is represented by $\mathbf{v} E \in M_{1 \times 2}(\overline K)$. We will show the following proposition (compare to \cite{CP11}, Theorem 4.4.2):

\begin{proposition} \label{proposition: Ext}
Suppose that $\mathbf{u}_1,\ldots,\mathbf{u}_m \in M_{n \times 1}(\bC_\infty)$ with $\Exp_\rho^{\otimes n}(\mathbf{u}_i)=\mathbf{v}_i \in M_{n \times 1}(\overline K)$. If $\pi^n_\rho, \mathbf{u}_{1,n},\ldots,\mathbf{u}_{m,n}$ are linearly independent over $K$, then the classes of $\mathcal{X}_{\mathbf{v}_i} (1 \leq i \leq n)$ in $\text{Ext}^1_\mathcal{T}(\mathbf 1,\mathcal{X}_\rho^{\otimes n})$ are linearly independent over $K$.
\end{proposition}

\begin{proof}
The proof follows closely that of \cite{CP11}, Theorem 4.4.2. Suppose that there exist $e_1,\ldots,e_m \in K$ not all zero so that $N=e_{1*} X_1+\ldots+e_{m*} X_m$ is trivial in $\text{Ext}^1_\mathcal{T}(\mathbf 1,\mathcal{X}_\rho^{\otimes n})$. We can suppose that for $1 \leq i \leq m$, $e_i$ belongs to $A$ and is represented by $M_i \in M_{2 \times 2}(\overline K[t])$. Then the extension $N=e_{1*} X_1+\ldots+e_{m*} X_m$ is represented by $h_{\mathbf{v}_1} M_1+\ldots+h_{\mathbf{v}_m} M_m$ and we have
\begin{align*}
\Phi_N &=\begin{pmatrix}
\Phi_\rho^{\otimes n} & 0  \\
\sum_{i=1}^m h_{\mathbf{v}_i} M_i & 1
\end{pmatrix} \in M_{3 \times 3}(\overline K[t]), \\
\Psi_N &=\begin{pmatrix}
\Psi_\rho^{\otimes n} & 0  \\
(\sum_{i=1}^m g_{\mathbf{v}_i} M_i)\Psi_\rho^{\otimes n}  & 1
\end{pmatrix} \in M_{3 \times 3}(\mathbb{T}).
\end{align*}
Since this extension is trivial in $\text{Ext}^1_\mathcal{T}(\mathbf 1,\mathcal{X}_\rho^{\otimes n})$, there exists a matrix 
\begin{align*}
\gamma &=\begin{pmatrix}
\text{Id}_2 & 0  \\
(\gamma_1,\gamma_2)  & 1
\end{pmatrix} \in M_{3 \times 3}(\overline K[t])
\end{align*}
such that $\gamma^{(-1)} \Phi_N=\text{diag}(\Phi_\rho^{\otimes n},1) \gamma$. By \cite{Pap08}, Section 4.1.6, there exists 
\begin{align*}
\delta &=\begin{pmatrix}
\text{Id}_2 & 0  \\
(\delta_1,\delta_2)  & 1
\end{pmatrix} \in M_{3 \times 3}(\Fq(t))
\end{align*}
such that $\gamma \Psi_N=\text{diag}(\Psi_\rho^{\otimes n},1) \delta$. It follows that 
	\[ (\gamma_1,\gamma_2)+(\sum_{i=1}^m g_{\mathbf{v}_i} M_i)=(\delta_1,\delta_2) (\Psi_\rho^{\otimes n})^{-1}.\]
Specializing the first coordinates at $t=\theta$ and recalling that by Lemma \ref{lemma:endo}, we have $(M_i)_{1,1}(\theta) \in K^\times$ and $(M_i)_{2,1}(\theta)=0$, we obtain
	\[ \gamma_1(\theta)+\sum_{i=1}^m g_{\mathbf{v}_i,1}(\theta) (M_i)_{1,1}(\theta)=\delta_1(\theta) [(\Psi_\rho^{\otimes n})^{-1}]_{1,1}(\theta)+\delta_2(\theta) [(\Psi_\rho^{\otimes n})^{-1}]_{2,1}(\theta).\]
By Lemma \ref{L:ulemma}, we have $g_{\mathbf{v}_i,1}(\theta)=\mathbf{u}_{i,n}-\mathbf{v}_{i,n}$ and $[(\Psi_\rho^{\otimes n})^{-1}]_{1,1}(\theta),[(\Psi_\rho^{\otimes n})^{-1}]_{2,1}(\theta) \in K \pi_\rho^n$. Since $(M_i)_{1,1}(\theta) \in K^\times$ (see Section \ref{sec: endos}), we get a non trivial $\overline K$-linear relation between $1, \mathbf{u}_{1,n},\ldots,\mathbf{u}_{m,n}, \pi_\rho^n$. By \cite{Gre17b}, page 29 (proof of Theorem 7.1), it implies a non trivial $K$-linear relation between $\mathbf{u}_{1,n},\ldots,\mathbf{u}_{m,n}, \pi_\rho^n$. Thus we get a contradiction. 

The proof is finished.
\end{proof}

\subsection{An application of Hardouin's work} \label{sec: Galois groups} ${}$\par

We now apply Hardouin's work to our context to determine Galois groups of $t$-motives of our interest. We work with the neutral Tannakian category of $t$-motives $\mathcal T$ over $F=\Fq(t)$ endowed with the fiber functor $\omega:M \mapsto H_{\text{Betti}}(M)$. We consider the irreducible object $\mathcal Y=\mathcal{X}^{\otimes n}_\rho$ and extensions of $\mathbf 1$ by $\mathcal{X}^{\otimes n}_\rho$. Hardouin's work turns out to be a powerful tool and allows us to prove the proposition below which generalizes the results of Papanikolas \cite{Pap08} (for the Carlitz module $C$) and Chang-Yu \cite{CY07} (for the tensor powers $C^{\otimes n}$ of the Carlitz module). 

\begin{proposition} \label{proposition: dimension}
Let $n \geq 1$ be an integer and $b$ be an element in $B$. Then the unipotent radical of $\Gamma_{\mathcal{X}_n(b)}$ is equal to the $\Fq[t]$-vector space $\Fq[t]^2$ of dimension 2. In particular,
	\[ \dim \Gamma_{\mathcal{X}_n(b)}=\dim \Gamma_{\mathcal{X}_\rho^{\otimes n}}+2=4.\]
\end{proposition}

\begin{proof}
We claim that the assumptions of Theorem \ref{theorem:Hardouin} are satisfied for the $t$-motive $\mathcal{X}^{\otimes n}_\rho$ since 
\begin{itemize}
\item[1.] By Proposition \ref{proposition: Galois group}, the Galois group $\Gamma_{\mathcal{X}_\rho^{\otimes n}}$ of $\mathcal{X}^{\otimes n}_\rho$ is a torus. Thus every $\Gamma_{\mathcal{X}_\rho^{\otimes n}}$ is completely reducible (compare to \cite{CP11}, Corollary 3.5.7).

\item[2.] It is clear that the center of $\Gamma_{\mathcal{X}_\rho^{\otimes n}}$, which is itself, contains $\mathbb{G}_{m,\Fq(t)}$.

\item[3.] The action of $\mathbb{G}_{m,\Fq(t)}$ on $H_{\text{Betti}}(\mathcal{X}^{\otimes n}_\rho)$ is isotypic. In fact, the weights are all equal to $n$.

\item[4.] The Galois groups of $t$-motives are reduced (see \cite{Pap08} and also \cite{HJ16}, Proposition 6.2 for $A$-motives).
\end{itemize}

We apply Theorem \ref{theorem:Hardouin} to the $t$-motive $\mathcal{X}_n(b)$ which is an extension of $\mathbf{1}$ by $\mathcal{X}^{\otimes n}_\rho$. Thus there exists a sub-object $\mathcal{V}$ of $\mathcal{X}^{\otimes n}_\rho$ such that $\mathcal{X}_n(b)/\mathcal{V}$ is a trivial extension of $\mathbf{1}$ by $\mathcal{X}^{\otimes n}_\rho/\mathcal{V}$. By \cite{Gre17b}, Lemma 7.2, we know that $\mathcal{X}^{\otimes n}_\rho$ is simple. Thus either $\mathcal{V}=0$ or $\mathcal{V}=\mathcal{X}^{\otimes n}_\rho$.

We claim that $\mathcal{V}=\mathcal{X}^{\otimes n}_\rho$. In fact, suppose that $\mathcal{V}=0$. We deduce that $\mathcal{X}_n(b)$ is a trivial extension of $\mathbf{1}$ by $\mathcal{X}^{\otimes n}_\rho$. It follows that $\pi_\rho^n$ and $\zeta_\rho(b,n)$ are linearly dependent over $K$. We get a contradiction by Proposition \ref{proposition: Ext} and Theorem \ref{T:LinearIndependence}.

Since $\mathcal{V}=\mathcal{X}^{\otimes n}_\rho$, by Theorem \ref{theorem:Hardouin}, the unipotent radical of the Galois group $\Gamma_{\mathcal{X}_n(b)}$ is equal to $H_{\text{Betti}}(M)(\mathcal{X}^{\otimes n}_\rho)$ that is an $\Fq[t]$-vector space of dimension 2. The Theorem follows immediately.
\end{proof}

As a consequence, we obtain a generalization of \cite{CY07}, Theorem 4.4.

\begin{corollary}
Let $n \geq 1$ be an integer. Then for any $b \in B$, the quantities $\pi_\rho$ and $\zeta_\rho( b,n)$ are algebraically independent over $\overline K$.
\end{corollary} 

\begin{proof}
This is a direct consequence of Theorem \ref{proposition: dimension}, Theorem \ref{theorem: Papanikolas} and the calculations from section \ref{sec: period calculations}.
\end{proof}

We obtain the following theorem which could be considered as a generalization of \cite{CP11}, Theorem 5.1.5 in our context. 

\begin{theorem} \label{theorem: algebraic independence log}
Suppose that $\mathbf{u}_1,\ldots,\mathbf{u}_m \in M_{n \times 1}(\bC_\infty)$ with $\Exp_\rho^{\otimes n}(\mathbf{u}_i)=\mathbf{v}_i \in M_{n \times 1}(\overline K)$. If $\pi^n_\rho, \mathbf{u}_{1,n},\ldots,\mathbf{u}_{m,n}$ are linearly independent over $K$, then they are algebraically independent over $\overline K$.
\end{theorem}

\begin{proof}
By Proposition \ref{proposition: Ext}, we deduce that the classes of $\mathcal{X}_{\mathbf{v}_i} (1 \leq i \leq n)$ in $\text{Ext}^1_\mathcal{T}(\mathbf 1,\mathcal{X}_\rho^{\otimes n})$ are linearly independent over $K$. By Corollary \ref{cor:Hardouin}, the unipotent part of the Galois group of the direct sum $\mathcal{X}_{\mathbf{v}_1} \oplus \ldots \oplus \mathcal{X}_{\mathbf{v}_n}$ is of dimension $2n$. Thus the Theorem follows immediately from Theorem \ref{theorem: Papanikolas}.
\end{proof}


\section{Algebraic relations among Anderson's zeta values} \label{sec: Anderson zeta values}

\subsection{Direct sums of $t$-motives} ${}$\par

Let $m \in \mathbb N, m \geq 1$. To study Anderson's zeta values $\zeta_\rho(b,n)$ for $1 \leq n \leq m$ and $\pi_\rho$ simultaneously, we set 
\begin{align*}
\mathcal S:=\{n \in \mathbb N: 1 \leq n \leq m \text{ such that } p \nmid n \text{ and } (q-1) \nmid n \},
\end{align*}
and consider the direct sum $t$-motive
	\[ \mathcal{X}(b):=\bigoplus_{n \in \mathcal S} \mathcal{X}_n(b) \]
and define block diagonal matrices
	\[ \Phi(b):=\bigoplus_{n \in \mathcal S} \Phi_n(b), \quad \Psi(b):=\bigoplus_{n \in \mathcal S} \Psi_n(b).\]
Then $\Phi(b)$ represents multiplication by $\sigma$ on $\mathcal{X}(b)$ and $\Psi(b)$ is a rigid analytic trivialization of $\Phi(b)$. We would like to understand the Galois group $\Gamma_{\mathcal{X}(b)}$ of the $t$-motive $\mathcal{X}(b)$ and to calculate the dimension of this Galois group. 

We first have 
	\[ \Gamma_{\mathcal{X}(b)} \subseteq \bigoplus_{n \in \mathcal S} \Gamma_{\mathcal{X}_n(b)}=\bigoplus_{n \in \mathcal S}  \begin{pmatrix}
\textrm{Res}_{K/\Fq[t]} \mathbb{G}_{m,K} & 0\\
* &  1
\end{pmatrix}\]
For $n=1$, the $t$-motive $\mathcal{X}_1(b)$ contains $\rho$. It follows that $\rho$ is also contained in $\mathcal{X}(b)$. We consider $\mathcal T_{\mathcal{X}(b)}$ and $\mathcal T_{\mathcal{X}_1(b)}$, the strictly full Tannakian subcategories of the category $\mathcal T$ of $t$-motives which are generated by $\mathcal{X}(b)$ and $\mathcal{X}_1(b)$ respectively. Thus we get a functor from $\mathcal T_{\mathcal{X}_1(b)}$ to $\mathcal T_{\mathcal{X}(b)}$. By Tannakian duality, we have a surjective map of algebraic groups over $\Fq(t)$
	\[ \pi:\Gamma_{\mathcal{X}(b)} \twoheadrightarrow \Gamma_{\mathcal{X}_1(b)}=\textrm{Res}_{K/\Fq[t]} \mathbb{G}_{m,K} \]
where we have the last equality by Proposition \ref{proposition: Galois group}. By Equation \eqref{eq: Galois Frob}, this map $\pi$ is in fact the projection on the upper left-most corner of elements of $\Gamma_{\mathcal{X}(b)}$. We denote by $U(b)$ the kernel of $\pi$. It follows that $U(b)$ is contained in the unipotent group
	\[ U:=\bigoplus_{n \in \mathcal S}  \begin{pmatrix}
\textrm{Id}_2 & 0\\
* &  1
\end{pmatrix}.\]

We prove the following result similar to \cite{CY07}, Section 4.3.
\begin{proposition} \label{prop: unipotent direct sum}
We keep the previous notation. Then we have
	\[ U(b)=\bigoplus_{n \in \mathcal S}  \begin{pmatrix}
\textrm{Id}_2 & 0\\
* &  1
\end{pmatrix}.\]
\end{proposition}

\begin{proof}
In fact, the strategy of Chang-Yu (see \cite{CY07}, Section 4.3) based on a weight argument indeed carries over without much modification.  For completeness, we sketch a proof of this Proposition.

We introduce a $\mathbb G_{m,\Fq(t)}$-action on $U(b)$ and on the direct sum of unipotent groups 
	\[ U=\bigoplus_{n \in \mathcal S}  \begin{pmatrix}
\textrm{Id}_2 & 0\\
* &  1
\end{pmatrix}.\] 
On the matrix indexed by $n \in \mathcal S$, it is defined by 
	\[ a \cdot \begin{pmatrix}
\text{Id}_2 & 0\\
\mathbf u &  1
\end{pmatrix} \mapsto \begin{pmatrix}
\text{Id}_2 & 0\\
a^n \mathbf u &  1
\end{pmatrix}, \quad a \in \mathbb G_{m,\Fq(t)}. \]
Note that this action on $U(b)$ agrees with the conjugation of $\mathbb G_{m,\Fq(t)}$ on $U(b)$.

For each $n \in \mathcal S$, we recall that 
	\[ \Gamma_{\mathcal{X}_n(b)}=\begin{pmatrix}
\textrm{Res}_{K/\Fq[t]} \mathbb{G}_{m,K} & 0\\
* &  1
\end{pmatrix}.\]
We denote by $U_n(b)$ the unipotent part of this Galois group. Thus
	\[ U_n(b)=\begin{pmatrix}
\text{Id}_2 & 0\\
* &  1
\end{pmatrix}\]
and we have a short exact sequence 
	\[ 1 \rightarrow U_n(b) \rightarrow \Gamma_{\mathcal{X}_n(b)} \rightarrow \textrm{Res}_{K/\Fq[t]} \mathbb{G}_{m,K} \rightarrow 1. \]
Since $\mathcal{X}_n(b)$ is contained in $\mathcal{X}(b)$, by Tannakian duality, we obtain a commutative diagram
\begin{align*}
\begin{CD}
1 @>>> U(b) @>>> \Gamma_{\mathcal{X}(b)} @>>> \textrm{Res}_{K/\Fq[t]} \mathbb{G}_{m,K} @>>> 1 \\
& & @V{\varphi_n}VV @V{\varphi_n}VV @V{\chi_n}VV \\
1 @>>> U_n(b) @>>> \Gamma_{\mathcal{X}_n(b)} @>>> \textrm{Res}_{K/\Fq[t]} \mathbb{G}_{m,K} @>>> 1.
\end{CD}
\end{align*}
Here the middle vertical arrow is surjective by Tannakian duality and the map $\chi_n$ is the character $a \mapsto a^n$. We deduce that the induced map $U(b) \rightarrow U_n(b)$ is also surjective.

We suppose now that $U(b)$ is of codimension $r>0$ in $U$. We identify $U$ with the product 
	\[ U \simeq \prod_{n \in \mathcal S} \mathbb G_{a,\Fq(t)}^2. \]
Chang and Yu proved that there exist an integer $n \in \mathcal S$ and a set $J$ of $r$ double indices $ij$ with $i \in \mathcal S$ and $j \in \{1,2\}$ such that if we denote by $W_{(J)}$ the linear subspace of $U$ of codimension $r$ consisting of points $(x_{ij})$ satisfying $x_{ij}=0$ whenever $ij \in J$, then $W_{(J)} \cap U_n(b) \subsetneq U_n(b)$ and the composed map 
	\[ f_n:W_{(J)} \hookrightarrow U(b) \overset{\varphi_n}{\longrightarrow} U_n(b) \]
is surjective.

Recall that for $k \in \mathcal S$, the action of $\mathbb G_{m,\Fq(t)}$ on $U_k(b)$ is of weight $k$. Since $p \nmid n$, by \cite{CY07}, Lemma 4.7, $f_n$ maps $W_{(J)} \cap U_k(b)$ to zero for all $k \neq n$ in $\mathcal S$. Thus it maps $W_{(J)} \cap U_n(b)$ onto $U_n(b)$ which has strictly greater dimension. We obtain a contradiction. 

As a consequence, we get $U(b)=U$ as required. The proof is complete.
\end{proof}

\subsection{Algebraic relations among Anderson's zeta values} ${}$\par

As an immediate consequence of Proposition \ref{prop: unipotent direct sum}, we see that the radical unipotent of $\Gamma_{\mathcal{X}(b)}$ is what we expect and thus of dimension $2|\mathcal S|$. By Theorem \ref{theorem: Papanikolas}, we deduce the following theorem.

\begin{theorem} \label{theorem:main1}
Let $b \in B$. Then the elements of the following set
\begin{align*}
\{\pi_\rho\} \cup \{\zeta_\rho(b,n), 1 \leq n \leq m \text{ such that } p \nmid n \text{ and } (q-1) \nmid n \}.
\end{align*} 
are algebraically independent over $\overline K$.
\end{theorem}

We present a slight generalization of the above theorem by taking account of the $p$-power relations. Let $\{b_1,\ldots,b_h\}$ be a $K$-basis of $H$ with $b_i \in B$. Since the extension $H/K$ is separable, it follows that for any $b \in B$, we can write $b=a_1 b_1^{p^m} + \ldots + a_h b_h^{p^m}$ with $a_1,\ldots,a_h \in K$. Thus we get
\begin{align*}
\zeta_\rho(b,p^m n) =\sum_{I \subseteq A} \frac{\sigma_I(b)}{u_I^{p^m n}} =\sum_{i=1}^h a_i \left( \sum_{I \subseteq A} \frac{\sigma_I(b_i)}{u_I^n} \right)^{p^m} =\sum_{i=1}^h a_i \zeta_\rho(b_i,n)^{p^m}.
\end{align*}

\begin{theorem} \label{theorem:main2}
Let $\{b_1,\ldots,b_h\}$ be a $K$-basis of $H$ with $b_i \in B$. We consider the following set
\begin{align*}
\mathcal A=\{\pi_\rho\} \cup \{\zeta_\rho(b_i,n): 1 \leq i \leq h, 1 \leq n \leq m \text{ such that } q-1 \nmid n \text{ and } p \nmid n \}.
\end{align*}
Then the elements of $\mathcal A$ are algebraically independent over $\overline K$.
\end{theorem}

\begin{proof}
The proof of Theorem \ref{theorem:main2} follows identically to that of Theorem \ref{theorem:main1}.
\end{proof}


\section{Algebraic relations among Goss's zeta values} \label{sec: Goss zeta values}

In this Section, we investigate algebraic relations among Goss's zeta values. This Section owes its very existence to B. Anglès. In particular, the proofs of Proposition \ref{prop: Goss} and Corollary \ref{cor: Goss zeta values} are due to him. For more details about the theory of $L$ series and Goss's zeta values, we refer the interested reader to \cite{Gos96}, Section 8. 

\subsection{Goss's map} ${}$\par

We set $\pi:=t/y$ which is a uniformizer  of $K_\infty$. Set $\pi_1=\pi$, and for $n\geq 2$, choose $\pi_n\in \overline{K}_\infty^\times$ such that $\pi_n^n=\pi_{n-1}$. If $z\in \mathbb Q$, $z=\frac{m}{n!}$ for some $m\in \mathbb Z, n\geq 1$, we set
	\[ \pi^z := \pi_n^m. \]
Let $\overline{\mathbb F}_q$ be the algebraic closure of $\mathbb F_q$ in $\overline{K}_\infty$, and let 
	\[ U_\infty:=\left\{ x\in \overline{K}_\infty, v_\infty(x-1)>0\right\}. \] 
Then $\overline{K}_\infty^\times= \pi^{\mathbb Q}\times \overline{\mathbb F}_q^\times\times U_\infty$. Therefore, if $x\in \overline{K}_\infty^\times$, one can write in a unique way:
 \[ x=\pi^{v_\infty(x)} \sgn (x) \langle x \rangle, \quad \sgn(x)\in \overline{\mathbb F}_q^\times, \langle x \rangle \in U_\infty. \]

Let $I\in \mathcal I(A)$, then there exists an integer $h \geq 1$ such that $I^h=xA, $ $x\in K^\times$. We set $\langle I \rangle:= \langle x \rangle^{\frac{1}{h}}\in U_\infty$. Then one shows (see \cite{Gos96}, Section 8.2) that the map called Goss's map
\begin{align*}
[\cdot]_A: \mathcal I(A) &\rightarrow \overline{K}_\infty^\times \\
I &\mapsto \langle I \rangle\pi^{-\frac{\deg I}{d_\infty}}
\end{align*}
is a group homomorphism such that
	\[ \forall x\in K^\times, \quad [xA]_A=\frac{x}{\sgn(x)}. \]
Observe that for all $I\in \mathcal I(A)$, we have $\sgn \left([I]_A\right)=1$.

Let $E/K$ be a finite extension, and let $O_E$ be the integral closure of $A$ in $E$. Let $\mathcal I(O_E)$ be the group of non-zero fractional ideals of $O_E$. We denote by $N_{E/K}:  \mathcal I(O_E)\rightarrow \mathcal I(A)$ the group homomorphism such that, if $\frak P$ is a maximal ideal of $O_E$ and $P=\frak P\cap A$, we have
	\[ N_{E/K}(\frak P)=P^{\left[\frac{O_E}{\frak P}: \frac{A}{P}\right]}. \]
Note that if $\frak P=xO_E, x\in E^\times$, then $N_{E/K} (\frak P) =N_{E/K} (x)A$ where $N_{E/K}: E\rightarrow K$ also denotes the usual norm map.

\subsection{Goss's zeta functions and Goss's zeta values}\label{Z}${}$ \par

We recall the definition of Goss's zeta functions introduced in \cite{Gos96}, Chapter 8. Let $\mathbb S_\infty=\mathbb C_\infty^\times \times \mathbb Z_p$ be the Goss ``complex plane". The group action of $\mathbb S_\infty$ is written additively. Let $I\in \mathcal I(A)$ and $s=(x;y)\in \mathbb S_\infty$, we set
	\[ I^s:= \langle I \rangle^yx^{\deg I}\in \mathbb C_\infty^\times. \]
We have a natural injective group homomorphism: $\mathbb Z\rightarrow \mathbb S_\infty, j\mapsto s_j=\left(\pi^{-\frac{j}{d_\infty}},j\right)$. Observe that $I^{s_j}=[I]^j_A$.

Let $E/K$ be a finite extension, and let $O_E$ be the integral closure of $A$ in $E$. Let $\frak I$ be a non-zero ideal of $E$. We have
	\[ \forall j\in \mathbb Z, \quad N_{E/K}(\frak I)^{s_j}=\left[\frac{O_E}{\frak I}\right]_A^j. \]
Letting $s\in \mathbb S_\infty$,  the following sum converges in $\mathbb C_\infty$ (see \cite{Gos96}, Theorem 8.9.2):
	\[ \zeta_{O_E}(s):=\sum_{d \geq 0} \sum_{\substack{\frak I\in \mathcal I(O_E), \frak I\subset O_E, \\ \deg(N_{E/K}(\frak I))=d}} N_{E/K}(\frak I)^{-s}. \]
The function $\zeta_{O_E}:\mathbb S_\infty\rightarrow \mathbb C_\infty$ is called the \emph{zeta function attached to $O_E$ and $[\cdot]_A$.} Observe that
	\[ \forall j\in \mathbb Z, \quad \zeta_{O_E}(j):= \zeta_{O_E}(s_j)=\sum_{d\geq 0} \sum_{\substack{\frak I\in \mathcal I(O_E), \frak I\subset O_E,\\ \deg(N_{E/K}(\frak I))=d}} \left[\frac{O_E}{\frak I}\right]_A^{-j}. \]
In particular,
\[ \zeta_{O_E}(1)=\prod_{\frak P} \left(1-\frac{1}{\left[\frac{O_E}{\frak P}\right]_A}\right)^{-1}\in \overline{K}_\infty^\times, \]
where $\frak P$ runs through the maximal ideals of $O_E$.

Recall that $\rho: A\rightarrow B\{\tau\}$ is the sign-normalized rank one Drinfeld module given in Section \ref{S:Review of Tensor Powers}. By \cite{ANDTR17a}, Proposition 2.1, the following product converges to an element in $U_\infty \cap K_\infty^\times$:
	\[ L_A(\rho/O_E):= \prod_{\frak P} \frac{[\text{Fitt}_A(O_E/\frak P]_A}{[\text{Fitt}_A(\rho(O_E/\frak P))]_A}\]
where $\frak P$ runs through the maximal ideals of $O_E$.

We have the following crucial fact (see \cite{ANDTR17a}, Proposition 3.4) which provides a deep connection between the special $L$-values and the Goss's zeta value at $1$.

\begin{proposition}
Let $E/K$ be a finite extension such that $H\subset E$. Then
	\[ L_A(\rho/O_E)=\zeta_{O_E}(1). \]
\end{proposition}


\subsection{Relations with Anderson's zeta values} ${}$ \par

Let $z$ be an indeterminate over $K_\infty$, and recall that $\mathbb T_z(K_\infty)$ denotes the Tate algebra in the variable $z$ with coefficients in $K_\infty$. Recall that
	\[ H_\infty=H\otimes_KK_\infty, \]
	\[ \mathbb T_z(H_\infty)= H\otimes_K\mathbb T_z(K_\infty). \]
For $n\in \mathbb Z$, we set
	\[ Z_{B}(n;z)=\sum_{d \geq 0} \sum_{\substack{\frak I\in \mathcal I(B), \frak I\subset B,\\ \deg(N_{E/K}(\frak I))=d}} \left[\frac{O_E}{\frak I}\right]_A^{-n} z^d. \]
Then, by \cite{Gos96}, Theorem 8.9.2,  for all $n\in \mathbb Z, $ $Z_B(n;.)$ defines an entire function on $\mathbb C_\infty$, and
	\[ \forall n\in \mathbb N, \quad Z_B(-n; z)\in A[z]. \]
Observe that
	\[ \forall n\in \mathbb Z, \quad Z_B(n; z)\in \mathbb T_z(K_\infty), \]
and
	\[ \forall n\geq 1, \quad Z_B(n; z)=\prod_{\frak P} \left(1-\frac {z^{\deg(N_{H/K}(\frak P))}}{\left[\frac{O_E}{\frak P}\right]^n_A}\right)^{-1}\in \mathbb T_z(K_\infty)^\times. \]
Finally, we note that
	\[ Z_B(n; 1)=\zeta_B(n). \]

Recall that $G={\rm Gal}(H/K)$. Then $G\simeq {\rm Gal}(H(z)/K(z))$ acts on $\mathbb T_z(H_\infty)$. We denote by $\mathbb T_z(H_\infty)[G]$ the non-commutative group ring where the commutation rule is given by
\[ \forall h,h'\in \mathbb T_z(H_\infty), \forall g,g'\in G, \quad hg.h'g'= hg(h') gg'. \]

Let $n\in \mathbb Z$. One can show (see \cite{ANDTR17a}, Lemma 3.5) that the following infinite sum converges in $\mathbb T_z({H}_\infty)[G]$:
	\[ \mathcal L(\rho/B; n; z):=\sum_{d \geq 0} \sum_{\substack{I\in \mathcal I(A) , I\subset A,\\ \deg I=d}} \frac{z^{\deg I}}{\psi(I)^n} \sigma_I. \]
Furthermore, for all $n\geq 1$, we have
	\[ \mathcal L(\rho/B; n; z)=\prod_{P} \left(1-\frac{z^{\deg P}}{\psi(P)^n}\sigma_P\right)^{-1}\in \left(\mathbb T_z(H_\infty)[G]\right)^\times \]
and for all $n\leq 0$,
	\[ \mathcal L(\rho/B; n; z)\in B[z][G]. \]
Note that 
	\[ \mathcal \zeta_\rho(.,n)=\mathcal L(\rho/B; n;1)\in (H_\infty[G])^\times. \]
We observe that $\mathcal L(\rho/B; n; z)$ induces  a $\mathbb T_z(K_\infty)$-linear map $\mathcal L(\rho/B; n; z): \mathbb T_z(H_\infty)\rightarrow \mathbb T_z(H_\infty)$. Since $\mathbb T_z(H_\infty)$ is a free $\mathbb T_z(K_\infty)$-module of rank $[H:K]$ (recall that $\mathbb T_z(K_\infty)$ is a principal ideal domain), $\det_{\mathbb T_z(K_\infty)}\mathcal L(\rho/B; n; z)$ is well-defined.
We also observe that $\zeta_\rho(.,n)$ induces a $K_\infty$-linear map $\zeta_\rho(.,n): H_\infty\rightarrow H_\infty$, and we denote by $\det_{K_\infty} \zeta_\rho(.,n)$ its determinant. Recall that $\ev: \mathbb T_z(H_\infty)\rightarrow H_\infty$ is the $H_\infty$-linear map given by
	\[ \forall f\in \mathbb T_z(H_\infty), \quad \ev(f)=f\mid_{z=1}. \]
Observe that, if $\{e_1, \ldots, e_h\}$ is a $K$-basis of $H/K$ (recall that $h=[H:K]$), then
	\[ H_\infty= \oplus_{i=1}^h K_\infty e_i, \]
	\[ \mathbb T_z(H_\infty)= \oplus _{i=1}^h \mathbb T_z(K_\infty) e_i. \]
We deduce that
	\[ {\det}_{K_\infty} \zeta_\rho(.,n)=\ev\left({\det}_{\mathbb T_z(K_\infty)}\mathcal L(\rho/B; n; z)\right). \]
By \cite{ANDTR17a}, Theorem 3.6, we have
	\[ {\det}_{\mathbb T_z(K_\infty)}\mathcal L(\rho/B; n; z)= Z_B(n;z). \]
In particular,
	\[ {\det}_{K_\infty} \zeta_\rho(.,n)=\zeta_B(n). \]
	
\subsection{Algebraic relations among Goss's zeta values} ${}$\par

The class number $\text{Cl}(A)$ of $A$ equals to the number of rational points $X(\Fq)$ on the elliptic curve $X$ and also to the degree of extension $[H:K]$. For a prime ideal $\frak p$ of $A$ of degree $1$ corresponding to an $\Fq$-rational point on $X$, we denote by $\mathfrak p_+$ the subset of elements in $\mathfrak p$ of sign $1$ and consider the sum (compare to \cite{GP18}, Section 	6 and \cite{Gre17b}, Section 6):
	\[ \zeta_A(\frak p,n)=\sum_{\substack{a \in \mathfrak{p}^{-1},\\ \sgn(a)=1}}  \frac{1}{a^n}, \quad n \in \mathbb N. \]
We will see that the sums $\zeta_A(\frak p,n)$ where $\frak p$ runs through the set $\mathcal P$ of prime ideals of $A$ of degree 1 are the elementary blocks in the study of Goss's zeta values on elliptic curves. For the rest of this Section, it will be convenient to slightly modify these sums as follows.

\begin{proposition} \label{prop: Goss}
Let $n \in \mathbb N$. For $\sigma \in G=\Gal(H/K)$, we set
	\[ \zeta_A(\sigma,n):= \sum_{d\geq 0} \sum_{\substack{I\in \mathcal I(A), I\subset A,\\ \deg(I)=d, \\ \sigma_I=\sigma}} \frac{1}{[I]^n_A}. \]
Then the elements $\zeta_A(\sigma,n)$ indexed by $\sigma \in G$ are algebraically independent over $\overline K$.
\end{proposition}

\begin{proof}
Let $\sigma \in \Gal(H/K)$ and $\frak p$ be the corresponding ideal in $\mathcal P$ such that $\sigma_{\frak p}=\sigma$. We get
	\[ \zeta_A(\sigma,n) = \sum_{\substack{I\in \mathcal I(A), I\subset A,\\ \sigma_I=\sigma}} \frac{1}{[I]^n_A} =\frac{1}{[\frak p]^n_A} \sum_{\substack{a \in \mathfrak{p}^{-1},\\ \sgn(a)=1}} \frac{1}{a^n}=\frac{1}{[\frak p]^n_A} \zeta_A(\frak p,n), \]
and 
	\[ \sum_{\substack{I\in \mathcal I(A), I\subset A,\\ \sigma_I=\sigma}} \frac{1}{\psi(I)^n} =\frac{1}{\psi(\mathfrak{p})^n} \sum_{\substack{a \in \mathfrak{p}^{-1},\\ \sgn(a)=1}} \frac{1}{a^n}=\frac{1}{\psi(\mathfrak{p})^n} \zeta_A(\frak p,n). \]
Thus we obtain
	\[ \zeta_A(\sigma,n) = \frac{\psi(\mathfrak{p})^n}{[\mathfrak{p}]^n_A} \sum_{\substack{I\in \mathcal I(A), I\subset A,\\ \sigma_I=\sigma}} \frac{1}{\psi(I)^n}. \]
Note that $\frac{\psi(\mathfrak{p})^n}{[\mathfrak{p}]^n}$ belongs to $\overline K^\times$. It follows that for $b \in B$, we can express 
	\[ \zeta_\rho(b,n)=\sum_{\sigma \in G} a_\sigma(b) \, \zeta_A(\sigma,n) \]
with some coefficients $a_\sigma(b) \in \overline K$.

By Theorem \ref{theorem:main2}, if $b_1,\ldots,b_h \in B$ is a $K$-basis of $H$, then the elements $\zeta_\rho(b_i,n)$ $(1 \leq i \leq h)$ are algebraically independent over $\overline K$. By the above discussion and the fact that 
	\[ |\, \zeta_\rho(b_i,n), 1 \leq i \leq h \,|=|\, \zeta_A(\sigma,n), \sigma \in G \,|=[H:K], \]
the Proposition follows immediately.
\end{proof}

Let $U$ be the $p$-Sylow subgroup of $G$  where $p$ is the characteristic of $\mathbb F_q$. We set  $\Delta :=G/U={\rm Gal}(F/K)$ where $F=H^{U}$. We write $p^s=|U|$ and set
	\[ \widehat{G}={\rm Hom}(G, \overline{\mathbb F}_q^\times)={\rm Hom}(\Delta, \overline{\mathbb F}_q^\times)\simeq \Delta, \]
with $|\Delta| \in \mathbb Z_p^\times$.

For $\delta \in \Delta,$ we set
	\[ Z(n,\delta)=\sum_{\substack{I\in \mathcal I(A), I\subset A, \\ (I, F/K)=\delta}}\frac{1}{[I]_A^n} \in \overline{K}_\infty. \]
We see easily that
	\[ Z(n,\delta)= \sum_{ \sigma \equiv \delta \pmod{U}} \zeta_A(\sigma,n). \]
By Proposition \ref{prop: Goss}, $Z(n,\delta), \delta \in \Delta$ are algebraically independent over $\overline{K}$.

Let $\chi \in \widehat{G}$ and we consider the value at $1$ of Goss $L$-series attached to $\chi$ given by
	\[ L(n, \chi)=\sum_{\delta\in \Delta}\chi(\delta) Z(n,\delta) = \sum_{I\in \mathcal I(A), I\subset A}\frac{\chi((I,F/K))}{[I]_A^n} \]
where $(.,F/K)$ is the Artin map. It is clear that for all $\delta \in \Delta$,
	\[ Z(n,\delta) = \frac{1}{|\Delta|} \sum_{\chi \in \widehat{G}} \chi(\delta)^{-1} L(n, \chi). \]

The above discussion combined with Theorem \ref{theorem:main2} implies immediately a transcendental result for Goss's zeta values:
\begin{theorem} \label{theorem: Goss zeta values}
Let $m \in \mathbb N, m \geq 1$. Then the special values of Goss $L$-series
\begin{align*}
G_n=\{\pi_\rho\} \cup \{L(n, \chi): \chi \in \widehat{G}, 1 \leq n \leq m \text{ such that } q-1 \nmid n \text{ and } p \nmid n \}.
\end{align*}
are algebraically independent over $\overline K$.
\end{theorem}

As a direct consequence, we obtain the following corollary:

\begin{corollary} \label{cor: Goss zeta values}
Let $m \in \mathbb N, m \geq 1$. Let $L$ be an extension of $K$ such that $L \subset H$. We consider the following set
\begin{align*}
\mathcal G_L=\{\pi_\rho\} \cup \{\zeta_{O_L}(n): 1 \leq n \leq m \text{ such that } q-1 \nmid n \text{ and } p \nmid n \}.
\end{align*}
Then the elements of $\mathcal G_L$ are algebraically independent over $\overline K$.
\end{corollary}

\begin{remark}
1) When $L=K$, we have shown that $\zeta_A(1)$ is transcendental over $K$, which gives an affirmative answer to an old question of D. Goss \footnote{Personal communication in Spring 2016}.

2) When $L=H$, the above Theorem states that $\zeta_B(1)$ is transcendental over $K$. It answers positively to \cite{ANDTR20}, Problem 4.1 in this case. Note that our proof is highly nontrivial. 
\end{remark}

\begin{proof}[Proof of Corollary \ref{cor: Goss zeta values}]
Let $p^k$ be the exact power of $p$ that divides $[L:K]$ and let $N={\rm Gal} (F/F\cap L) \subseteq \Delta$.  We have (see for example \cite{Gos96}, Section 8.10):
	\[ \zeta_{O_L}(n) =\left(\prod_{\chi \in \widehat{G}, \chi(N)=\{1\}} L(n, \chi) \right)^{p^k}. \]
Thus Corollary \ref{cor: Goss zeta values} follows from Theorem \ref{theorem: Goss zeta values}.
\end{proof}



\end{document}